\documentclass[final]{siamltex}
\usepackage{amsbsy,amsmath,amsfonts,amssymb,graphicx,color}
\usepackage[all]{xy}
\usepackage{geometry} 

\date{\today}
\title{Controllability of 3D Low Reynolds Swimmers}
\author{J\'er\^ome Loh\'eac\thanks{E-mail: \texttt{jerome.loheac@iecn.u-nancy.fr}  and  \texttt{alexandre.munnier@iecn.u-nancy.fr}} \and Alexandre
Munnier$^{\ast}$\thanks{Both authors are with Institut \'Elie Cartan UMR 7502, Nancy-Universit\'e,
CNRS, INRIA, B.P.~239, F-54506 Vandoeuvre-l\`es-Nancy Cedex,
France, and INRIA
Nancy Grand Est, Projet CORIDA. Authors both supported by ANR CISIFS. Second author supported by ANR GAOS.}}
\begin{document}
\maketitle
\begin{abstract}
In this article, we consider a swimmer (i.e. a self-deformable body) immersed in a fluid, the flow of which is governed by the stationary Stokes equations. This model is relevant for studying the locomotion of microorganisms or micro robots for which the inertia effects can be neglected. Our first main contribution is to prove that any such microswimmer has the ability to track, by performing a sequence of shape changes, any given trajectory in the fluid. We show that, in addition,  this can be done by means of arbitrarily small body deformations that can be superimposed to any preassigned sequence of macro shape changes. Our second contribution is to prove that, when no macro deformations are prescribed, tracking is generically possible by means of shape changes obtained as a suitable combination of only four elementary deformations. Eventually, still considering finite dimensional deformations, we state results about the existence of optimal swimming strategies for a wide class of cost functionals.
\end{abstract}



\section{Introduction}
\subsection{Context}
Relevant models for the locomotion of microorganisms can be tracked back to the work of Taylor \cite{Taylor:1951aa}, Lighthill \cite{Lighthill:1952aa, Lighthill:1975aa}, and Childress \cite{Childress:1981aa}. 
Purcell explains in \cite{Purcell:1977aa} that these sort of animals are the order of a micron in size and they move around with a typical speed of 30 micron/sec. These data lead the flow regime to be characterized by a very small Reynolds number. For such swimmers, inertia effects play no role and the motion is entirely determined by the friction forces. 

In this article, the swimmer is modeled as a self deforming-body. By changing its shape, it set the surrounding fluid into motion and generates hydrodynamics forces used to propel and steer itself.  We are interested in investigating whether the microswimmer is able to control its trajectory by means of appropriate shape deformations (as real microorganisms do). This question has already be tackled in some specific cases. Let us mention \cite{Shapere:1989aa} (the authors study the motion of infinite cylinders with various cross sections and the swimming of spheres undergoing infinitesimal shape variations) and \cite{Alouges:2008aa} (in which the 1D controllability of a swimmer made of three spheres is investigated). 

Our contribution to this question is several folds. First, we give a definitive answer to the control problem in the general case: the swimmer we consider has any shape at rest (obtained as the image by a $C^1$ diffeomorphism of the unit ball) and can undergo any kind of shape deformations (as long as they can also be obtained as images of the unit ball by $C^1$ diffeomorphisms).  With these settings, we prove that the dynamical system governing the swimmer's motion in the fluid is controllable in the following sense: for any prescribed trajectory (i.e. given positions and orientations of the swimmer at every moment) there exists a sequence of shape changes that make him swim arbitrarily close to this trajectory. A somewhat surprising additional result is that this can be done by means of arbitrarily small shape changes which can be superimposed to any preassigned macro deformations (this is called the ability of {\it synchronized swimming} in the sequel). 
Second, when no macro deformations are prescribed (this is called {\it freestyle swimming} in the paper), we prove that the ability of tracking any trajectory is possible by means of shape changes obtained as an appropriate combination of only four elementary deformations (satisfying some generic assumptions). 
Third, we state a result about the existence of optimal swimming. 

Notice that the paper follows the lines of \cite{chamb_munnier_arxiv} in which the authors study the controllability of a swimmer in a perfect fluid.
\subsection{Modeling}
\subsubsection*{Kinematics}
We assume that the swimmer is the only immersed body in the fluid and that the fluid-swimmer system fills the whole space, identified with $\mathbf R^3$. Two frames are required in the modeling: The first one $\mathfrak E:=(\mathbf E_1,\mathbf E_2,\mathbf E_3)$ is fixed and Galilean and the second one $\mathfrak e:=(\mathbf e_1,\mathbf e_2,\mathbf e_3)$ is attached to the swimming body. At any moment, there exist a rotation matrix $R\in {\rm SO}(3)$ and a vector $\mathbf r\in\mathbf R^3$ such that, if $X:=(X_1,X_2,X_3)^\ast$ and $x:=(x_1,x_2,x_3)^\ast$ are the coordinates of a same vector in respectively $\mathfrak E$ and $\mathfrak e$, then the equality $X=Rx+\mathbf r$ holds. The matrix $R$ is meant to give also the {\it orientation} of the swimmer. The rigid displacement of the swimmer, on a time interval $[0,T]$ ($T>0$), is thoroughly described by the functions $t:[0,T]\mapsto R(t)\in{\rm SO}(3)$ and $t:[0,T]\mapsto \mathbf r(t)\in\mathbf R^3$, which are the unknowns of our problem. Denoting their time derivatives by $\dot R$ and $\dot{\mathbf r}$, we can define  the linear velocity $\mathbf v:=(v_1,v_2,v_3)^\ast\in\mathbf R^3$ and angular velocity vector $\boldsymbol\Omega:=(\Omega_1,\Omega_2,\Omega_3)^\ast\in\mathbf R^3$ (both in $\mathfrak e$) by respectively $\mathbf v:=R^\ast\dot{\mathbf r}$ and $\hat{\boldsymbol\Omega}:= R^\ast\dot R$, where for every vector $\mathbf u:=(u_1,u_2,u_3)^\ast\in\mathbf R^3$, $\hat{\mathbf u}$ is the unique skew-symmetric matrix satisfying $\hat{\mathbf u}x:=\mathbf u\times x$ for every $x\in\mathbf R^3$. 

\subsubsection*{Shape Changes}
Unless otherwise indicated, from now on all of the quantities will be expressed in the body frame $\mathfrak{e}$. In our modeling, the domains occupied by the swimmer are images of the closed unit ball $\bar B$ by $C^1$ diffeomorphisms, isotopic the identity, and tending to the identity at infinity, i.e. having the form ${\rm Id}+\vartheta$ where $\vartheta$ belongs to $D^1_0(\mathbf R^3)$ (the definitions of all of the function spaces are collected in the appendix, Section~\ref{SEC:diffeo}). With these settings, the shape changes over a time interval $[0,T]$ can be simply prescribed by means of functions $t\in[0,T]\mapsto \vartheta_t\in D^1_0(\mathbf R^3)$ lying in $W^{1,1}([0,T],D^1_0(\mathbf R^3))$. Then, denoting $\varTheta_t:={\rm Id}+\vartheta_t$, the domain occupied by the swimmer at every time $t\geq 0$  is the closed, bounded, connected set $\bar{\mathcal B}_t:=\varTheta_t(\bar B)$ (keep in mind that we are working in the frame $\mathfrak e$) and $\mathbf w_t:=\partial_t\varTheta(\varTheta^{-1})$ is the swimmer's Eulerian velocity of deformation. We shall denote $\Sigma:=\partial B$ the unit ball's boundary while $\Sigma_t:=\varTheta_t(\Sigma)$ will stand for the body-fluid interface.  The unit normal vector to $\Sigma_t$ directed toward the interior of $\mathcal B_t$ is $\mathbf n_t$ and the fluid fills the exterior open set $\mathcal F_t:=\mathbf R^3\setminus\bar{\mathcal B}_t$. 

\subsubsection*{The Flow}
The flow is governed by the stationary Stokes equations. They read (in the body frame $\mathfrak e$):
$$-\mu\Delta \mathbf u+\nabla p =0,\qquad\nabla\cdot\mathbf u=0\quad\text{in }\mathcal F_t\quad (t>0),$$
where $\mu$ is the viscosity, $\mathbf u$ the Eulerian velocity of the fluid and $p$ the pressure. These equations have to be complemented with the no-slip boundary conditions: $\mathbf u = \boldsymbol\Omega\times x+\mathbf v+\mathbf w_t$ on $\Sigma_t$. The linearity of these equations leads to introducing the elementary velocities and pressures $(\mathbf u_i, p_i)$ ($i=1,\ldots,6$) and $(\mathbf u_d,p_d)$, defined as the solutions to the Stokes equations with the boundary conditions $\mathbf u_i=\mathbf e_i\times x$ ($i=1,2,3$), 
$\mathbf u_i=\mathbf e_{i-3}$ ($i=4,5,6$) and $\mathbf u_d=\mathbf w_t$ on $\Sigma_t$. Then, the velocity $\mathbf u$ and the pressure $p$ can be decomposed as $\mathbf u=\sum_{i=1}^3\Omega_i\mathbf u_i+\sum_{i=4}^6 v_{i-3}\mathbf u_i+\mathbf u_d$ and $p=\sum_{i=1}^3\Omega_i p_i+\sum_{i=4}^6 v_{i-3} p_i+p_d$.
Notice that the pairs $(\mathbf u_i,p_i)$ ($i=1,\ldots,6$) and $(\mathbf u_d,p_d)$ are well-defined in the weighted Sobolev spaces $(W^1_0(\mathcal F_t))^3\times L^2(\mathcal F_t)$ (see the Appendix, Section~\ref{SEC:Stokes}).
 \begin{figure} \label{FIG_Kinematic}
 \centerline{\input{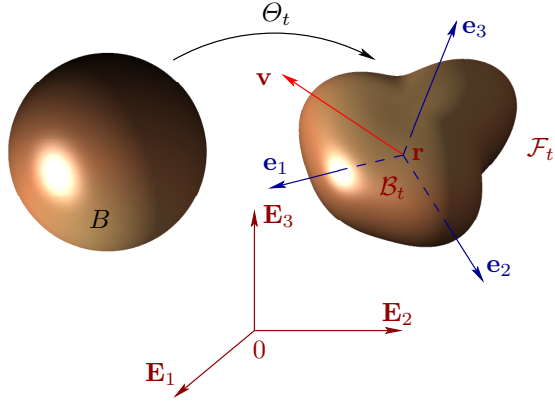}}
 \caption{Kinematics of the model: The Galilean frame $\mathfrak E:=(\mathbf E_j)_{1\leq j\leq 3}$ and the
body frame $\mathfrak e:=(\mathbf e_j)_{1\leq j\leq 3}$ with $\mathbf e_j=R\,\mathbf E_j$
 ($R\in{\rm SO}(3)$). Quantities are mostly expressed in
 the body frame. The domain of the body is $\bar{\mathcal B}_t$ at the time $t$ and $\mathcal B_t$ is the image of the
 unit ball $B$ by a diffeomorphism $\varTheta_t$. The open set $\mathcal F_t:=\mathbf R^3\setminus\bar{\mathcal B}_t$ is the domain of the fluid. The position of the swimmer is given by the vector $\mathbf r$ (in $\mathfrak E$) and its orientation by $R\in{\rm SO}(3)$. The vector $\mathbf v:=R^\ast\dot{\mathbf r}$ is the translational velocity (in $\mathfrak e$).}
 \end{figure}
\subsubsection*{Dynamics}
As already pointed out before, for microswimmers, the inertia effects are neglected in the modeling. Newton's laws reduce to $\int_{\Sigma_t}\mathbb T(\mathbf u,p)\mathbf n_t\times x\,{\rm d}\sigma =\mathbf 0$ (balance of angular momentum) and $\int_{\Sigma_t}\mathbb T(\mathbf u,p)\mathbf n_t\,{\rm d}\sigma =\mathbf 0$ (balance of linear momentum) where $\mathbb T(\mathbf u,p):=2\mu D(\mathbf u)-p{\rm Id}$ is the stress tensor of the fluid, with $D(\mathbf u):=(\nabla \mathbf u+\nabla \mathbf u^\ast)/2$. The stress tensor is linear with respect to $(\mathbf u,p)$ so it can be decomposed into $\mathbb T(\mathbf u,p)=\sum_{i=1}^3\Omega_i\mathbb T(\mathbf u_i,p_i)+\sum_{i=4}^6v_{i-3}\mathbb T(\mathbf u_i,p_i)+\mathbb T(\mathbf u_d,p_d)$. In order to rewrite Newton's laws in a short compact form, we introduce the
$6\times 6$ matrix $\mathbb M(t)$ whose entries
are $$M_{ij}(t):=\begin{cases}\int_{\Sigma_t}\mathbf
e_i\cdot(\mathbb T(\mathbf u_j,p_j)\mathbf n_t\times x){\rm
d}\sigma=\int_{\Sigma_t}(x\times \mathbf
e_i)\cdot\mathbb T(\mathbf u_j,p_j)\mathbf n_t{\rm d}\sigma& (1\leq i\leq 3,\, 1\leq j\leq 6);\\
\int_{\Sigma_t}\mathbf
e_{i-3}\cdot\mathbb T(\mathbf u_j,p_j)\mathbf n_t{\rm d}\sigma&(4\leq i\leq 6,\,1\leq j\leq
6);
\end{cases}
$$
 and $\mathbf N(t)$, the vector of $\mathbf R^6$ whose entries are 
 $$N_{i}(t):=\begin{cases}\int_{\Sigma_t}\mathbf
e_i\cdot(\mathbb T(\mathbf u_d,p_d)\mathbf n_t\times x){\rm
d}\sigma=\int_{\Sigma_t}(x\times \mathbf
e_i)\cdot\mathbb T(\mathbf u_d,p_d)\mathbf n_t{\rm d}\sigma& (1\leq i\leq 3);\\
\int_{\Sigma_t}\mathbf
e_{i-3}\cdot\mathbb T(\mathbf u_d,p_d)\mathbf n_t{\rm d}\sigma& (4\leq i\leq 6).
\end{cases}
$$
With these settings, Newton's laws take the convenient form $\mathbb M(t)(\boldsymbol\Omega,\mathbf v)^\ast+\mathbf N(t)=0$. Upon an integration by parts, the entries of
the matrix $\mathbb M(t)$ can be rewritten as $M_{ij}(\mathbf s):=2\mu\int_{\mathcal F_t}D(\mathbf u_i):D(\mathbf u_j){\rm d}x$, whence we deduce
that $\mathbb M(t)$ is symmetric and positive definite. We infer that the swimming motion is governed by the equation:
\begin{subequations}
\label{main_dynamics}
\begin{equation}
\label{dynamics}
\begin{pmatrix}\boldsymbol\Omega\\
\mathbf v\end{pmatrix}=-\mathbb M(t)^{-1}\mathbf N(t),\qquad (0\leq t\leq T).
\end{equation}
To determine the rigid motion in the fixed frame $\mathfrak E$, Equation \eqref{dynamics} has to be supplemented with the ODE:
\begin{equation}
\label{complement}
\frac{d}{dt}\begin{pmatrix}
R\\
{\mathbf r}
\end{pmatrix}=
\begin{pmatrix}
R\,\hat{\boldsymbol\Omega}\\
R\,\mathbf v
\end{pmatrix},\qquad(0<t<T),
\end{equation}
\end{subequations}
together with Cauchy data for $R(0)$ and $\mathbf r(0)$.
At this point, we can identify the control as being the function $t\in[0,T]\mapsto \vartheta_t\in D_0^1(\mathbf R^3)$. Notice that the dependence of the dynamics in the control is strongly nonlinear. Indeed $\vartheta_t$ describes the shape of the body and hence also the domain of the fluid in which are set the PDEs of the elementary velocity fields involved in the expressions of the matrices $\mathbb M(t)$ and $\mathbf N(t)$. 

Considering \eqref{main_dynamics}, we deduce as a first nice result:
\begin{proposition}
The dynamics of a microswimmer is independent of  the viscosity of the fluid.
Or, in other words, the same shape changes produce the same rigid displacement,
whatever the viscosity of the fluid is.
\end{proposition}
\begin{proof}
Let $(\mathbf u_j,p_j)$ be an elementary solution (as defined in the
modeling above) to the Stokes equations corresponding to the viscosity $\mu>0$,
then $(\mathbf u_j,(\tilde\mu/\mu) p_j)$ is the same elementary
solution corresponding to the viscosity $\tilde\mu>0$. Since the
Euler-Lagrange equation \eqref{main_dynamics} depends only on the Eulerian velocities $\mathbf
u_j$, the proof is completed.
\end{proof}

As a consequence of this Proposition we will set $\mu=1$ in the sequel.
\subsubsection*{Self-propelled constraints}
For our model to be more realistic, the swimmer's shape changes, instead of being preassigned, should be resulting from the interactions between some internal forces and the hydrodynamical forces exerted by the fluid on the body's surface. To do so, the dynamics \eqref{main_dynamics} should be supplemented with a set of equations (for instance PDEs of elasticity) allowing the shape changes to be computed from given internal forces. However, this would make the problem of locomotion much more involved and is beyond the scope of this paper. For weighted swimmers, this issue can be circumvented by adding constraints ensuring that the body's center of mass and moment of inertia are deformation invariant in the body frame. Unfortunately, massless microswimmers have no center of mass and their moment of inertia is always zero. 

To highlight the fact that constraints have still to be imposed to the shape changes for the control problem to make sense, consider the following result: 
\begin{proposition}
Let $\vartheta,\,\vartheta^\dag\in W^{1,1}([0,T],D^1_0(\mathbf R^3))$ be two control functions such that  $\varTheta:={\rm Id}+\vartheta$ and $\varTheta^\dag:={\rm Id}+\vartheta^\dag$ differ up to a rigid displacement on the unit sphere (more precisely, for every $t\in[0,T]$, there exists $(Q(t),\mathbf s(t))\in{\rm SO}(3)\times\mathbf R^3$ such that $(Q(0),\mathbf s(0))=({\rm Id},\mathbf 0)$ and $\varTheta^\dag_t|_{\Sigma}=Q(t)\varTheta_t|_\Sigma+\mathbf s(t)$). 
Then, denoting by $t\in[0,T]\mapsto (R(t),\mathbf r(t))\in{\rm SO}(3)\times\mathbf R^3$ a solution (if any) to System \eqref{main_dynamics} with Cauchy data $(R_0,\mathbf r_0)\in {\rm SO}(3)\times\mathbf R^3$, we get that the function $t\in[0,T]\mapsto (R^\dag(t),\mathbf r^\dag(t)):=(R(t)Q(t)^\ast,\mathbf r(t)-R(t)Q(t)^\ast\mathbf s(t))\in{\rm SO}(3)\times\mathbf R^3$ is also a solution with the same Cauchy data but control $\vartheta^\dag$. In particular $R^\dag(t)\varTheta_t^\dag+\mathbf r^\dag(t)=R(t)\varTheta_t+\mathbf r(t)$ for all $t\in[0,T]$ (i.e. the swimmer's global motion is the same in both cases).
\end{proposition}
\begin{proof} If we denote by $\mathbf u_i(t)$ ($i=1,\ldots,6$) (respectively $\mathbf u^\dag_i(t)$) the elementary velocity fields obtained with the control function $\vartheta$ (respectively $\vartheta^\dag$), it can be verified that $\mathbf u_i(t,x)=Q(t)^\ast\mathbf u_i^\dag(Q(t)x+\mathbf s(t))$ for every $t\in[0,T]$, every $x\in\mathcal F_t$ and every $i=1,\ldots,6$. We deduce that $\mathbb M(t)=\mathbb Q(t)^\ast\mathbb M^\dag(t)\mathbb Q(t)$
where the elements of $\mathbb M(t)$ (respectively $\mathbb M^\dag(t)$) have been computed with the elementary velocity fields $\mathbf u_i(t)$ (respectively $\mathbf u_i^\dag(t)$) and $\mathbb Q(t)\in{\rm SO}(6)$ is the bloc diagonal matrix ${\rm diag}(Q(t),Q(t))$. On the other hand, denoting respectively by $\mathbf w_t(x)=\partial_t\varTheta_t(\varTheta_t^{-1}(x))$ and $\mathbf w^\dag_t(x)=\partial_t\varTheta^\dag_t(\varTheta_t^{\dag-1}(x))$
the boundary velocity of the swimmer in both cases, we get the relation: $\mathbf w_t(x)+\boldsymbol\chi(t)\times x+\boldsymbol\zeta(t)=Q(t)^\ast\mathbf w_t^\dag(Q(t)x+\mathbf s(t))$ for all $t\in[0,T]$, where $\hat{\boldsymbol\chi}(t):=Q(t)^\ast \dot Q(t)$ and $\boldsymbol\zeta(t):=Q(t)^\ast\dot{\mathbf s}(t)$.
With obvious notation, we deduce that $\mathbf N(t)+\mathbb M(t)(\boldsymbol\chi(t),\boldsymbol\zeta(t))^\ast=\mathbb Q(t)^\ast\mathbf N^\dag(t)$.
If we set now $(\boldsymbol\Omega,\mathbf v)^\ast:=-\mathbb M(t)\mathbf N(t)$ and $(\boldsymbol\Omega^\dag,\mathbf v^\dag)^\ast:=-\mathbb M^\dag(t)\mathbf N^\dag(t)$, we get the identity $(\boldsymbol\Omega^\dag,\mathbf v^\dag)^\ast=\mathbb Q(t)(\boldsymbol\Omega-\boldsymbol\chi,\mathbf v-\boldsymbol\zeta)^\ast$. It suffices to integrate this relation, taking into account that $(Q(0),\mathbf s(0))=({\rm Id},\mathbf 0)$, to obtain the conclusion of the Proposition and to complete the proof.
\end{proof}

If we apply this proposition with $\vartheta$ constant in time (the boundary of the swimmer is $\varTheta(\Sigma)$ at any time), we deduce that any shape change which reduces to a rigid deformation $Q(t)x+\mathbf s(t)$ on the swimmer's boundary $\varTheta(\Sigma)$ will produce a displacement $(Q(t)^\ast,-Q(t)^\ast\mathbf s(t))$. But if we compute the global motion of the swimmer, we obtain $Q^\ast(t)(Q(t)\varTheta(x)+\mathbf s(t))-Q(t)^\ast\mathbf s(t)=\varTheta(x)$ for every $x\in\Sigma$ and every time $t$ which means that the swimmer is actually motionless (the rigid deformation of the swimmer's boundary is counterbalanced by its rigid displacement). To prevent this from happening, we add the following constraints to the deformations (inspired by the so-called {\it self-propelled constraints} for weighted swimmers, see for instance \cite{Chambrion:2011aa}):
\begin{equation}
\label{self-propelled-cond}
\int_{\Sigma}\varTheta_t(x)\,{\rm d}\sigma=\mathbf 0\quad(\text{for all }t\in[0,T])\text{ and}\quad
\int_{\Sigma}\partial_t\varTheta_t(x)\times \varTheta_t(x)\,{\rm d}\sigma=\mathbf 0\quad(\text{for a.e. }t\in[0,T]).
\end{equation}
About the existence of such deformations, we have in particular:
\begin{proposition}
\label{PROP:1_1}
For every function $\vartheta$ in $W^{1,1}([0,T],D^1_0(\mathbf R^3))$ such that $\int_{\Sigma}\varTheta_{t=0}(x)\,{\rm d}\sigma=\mathbf 0$, there exists a function $\vartheta^\dag$ in $W^{1,1}([0,T],D^1_0(\mathbf R^3))$ satisfying \eqref{self-propelled-cond} and an unique absolutely continuous rigid displacement $t\in[0,T]\mapsto (Q(t),\mathbf s(t))\in{\rm SO}(3)\times\mathbf R^3$ such that $Q(0)={\rm Id}$, $\mathbf s(0)=\mathbf 0$ and $\varTheta^\dag_t|_{\Sigma}=(Q(t)\varTheta_t+\mathbf s(t))|_{\Sigma}$ for every $t\in[0,T]$.
\end{proposition}

In other words, the proposition tells us that any function of $W^{1,1}([0,T],D^1_0(\mathbf R^3))$ satisfying the first equality of \eqref{self-propelled-cond} at $t=0$, can be made allowable (in the sense that it satisfies \eqref{self-propelled-cond}) when composed with a suitable rigid displacement on the unit sphere.
\begin{proof}
Define $\bar{\mathbf s}(t):=(1/4\pi)\int_\Sigma \varTheta_t\,{\rm d}\sigma$ (an absolutely continuous function on $[0,T]$) and $\bar\varTheta_t:=\varTheta_t-\bar{\mathbf s}(t)$ for every $t\in[0,T]$. The matrix $\mathbb J(t):=\int_\Sigma\|\bar\varTheta_t\|^2_{\mathbf R^3}{\rm Id}-\bar\varTheta_t\otimes\bar\varTheta_t{\rm d}\sigma$ is always definite positive since $(\mathbb J(t)x)\cdot x=\int_\Sigma \|\bar\varTheta_t\times x\|^2_{\mathbf R^3}{\rm d}\sigma$ for all $t\in[0,T]$ and all $x\in\mathbf R^3$. We can then define $\boldsymbol\chi(t):=\mathbb J(t)^{-1}\int_\Sigma\partial_t\bar\varTheta_t\times\bar\varTheta_t\,{\rm d}\sigma$ as a function of $L^1([0,T],\mathbf R^3)$. The absolutely continuous function $t\in[0,T]\mapsto Q(t)\in{\rm SO}(3)$ is obtained by solving the ODE $\partial_t Q(t)=Q(t)\hat{\boldsymbol\chi}(t)$ with Cauchy data $Q(0)={\rm Id}$ (we consider here a Carath\'eodory solution which is unique according to Gr\"onwall's inequality). Then, we set $\mathbf s(t):=-Q(t)\bar{\mathbf s}(t)$ for all $t\in[0,T]$. The function $\tilde\varTheta_t:=Q(t)\varTheta_t+\mathbf s(t)$ is in $W^{1,1}([0,T],C^1(\mathbf R^3)^3)$, satisfies \eqref{self-propelled-cond} but does not take its values in $D^1_0(\mathbf R^3)$ because $\tilde\varTheta_t(x)=Q(t)x+\mathbf s(t)+o(1)\neq x$ as $\|x\|_{\mathbf R^3}\to +\infty$. Let $\Omega$ and $\Omega'$ be large balls such that $\bigcup_{t\in[0,T]}\tilde\varTheta_t(\bar B)\subset\Omega$ and $\bar\Omega\subset\Omega'$ and consider a cut-off function $\xi$ valued in $[0,1]$ and such that  $\xi =1$ in $\Omega$ and $\xi=0$ in $\mathbf R^3\setminus\bar\Omega'$. To complete the proof, define $\varTheta^\dag$ as the flow associated with the Cauchy problem $\dot X(t,x)=\xi(x)\partial_t\tilde\vartheta_t(x)+(1-\xi(x))\partial_t\vartheta_t(x)$, $X(0,x)=\varTheta_{t=0}(x)$. \end{proof}
\begin{definition}
We denote by $\mathcal A$ the non-empty closed subset of $W^{1,1}([0,T],D^1_0(\mathbf R^3))$ consisting of all of the functions verifying \eqref{self-propelled-cond}.
\end{definition}
\subsection{Main results}
The first result ensures the well posedness of System \eqref{main_dynamics} and the continuity of the input-output mapping:

\begin{proposition}
\label{existence}
For any $T>0$, any function $\vartheta\in W^{1,1}([0,T],D^1_0(\mathbf R^3))$ (respectively of class $C^p$, $p=1,\ldots,+\infty,\omega$) and any initial data $(R(0),\mathbf r(0))\in{\rm SO}(3)\times\mathbf R^3$, System  \eqref{main_dynamics} admits a unique solution $t\in[0,T]\mapsto (R(t),\mathbf r(t))\in{\rm SO}(3)\times\mathbf R^3$ (in the sense of Carath\'eodory) absolutely continuous on $[0,T]$ (respectively of class $C^p$).

Let $(\vartheta_j)_{j\geq 1}\subset W^{1,1}([0,T],D^1_0(\mathbf R^3))$ be a sequence of controls converging to a function $\bar\vartheta$. 
Let a pair $(R_0,\mathbf r_0)\in{\rm SO}(3)\times\mathbf R^3$ be given and denote by $t\in[0,T]\mapsto(\bar R(t),\bar{\mathbf r}(t))\in{\rm SO}(3)\times\mathbf R^3$ the solution in $AC([0,T],{\rm SO}(3)\times\mathbf R^3)$ to System \eqref{main_dynamics}  
with control $\bar\vartheta$ and Cauchy data $(R_0,\mathbf r_0)$. Then, the unique solution $(R^j,\mathbf r^j)$ to System \eqref{main_dynamics} with control $\vartheta^j$ and Cauchy data $(R_0,\mathbf r_0)$ converges in $AC([0,T],{\rm SO}(3)\times\mathbf R^3)$ to $(\bar R,\bar{\mathbf r})$ as $j\to+\infty$. 
\end{proposition}

We denote by ${\rm M}(3)$ the Banach space of the $3\times 3$ matrices endowed with any matrix norm. The main result of this article addresses the controllability of System \eqref{main_dynamics}:
\begin{theorem}(Synchronized Swimming)
\label{main_theorem_cont}
Assume that the following data are given: (i) A  function $\bar\vartheta\in \mathcal A$ (the reference shape changes);
(ii) A continuous function $t\in[0,T]\mapsto (\bar R(t),\bar{\mathbf r}(t))\in{\rm SO}(3)\times\mathbf R^3$ (the reference trajectory to be followed). 
Then, for any $\varepsilon>0$, there exists a function $t\in[0,T]\mapsto \vartheta_t\in D^1_0(\mathbf R^3)$ (the actual shape changes) in $\mathcal A$, which can be chosen analytic, such that $\vartheta_0=\bar\vartheta_0$, $\vartheta_T=\bar\vartheta_T$ and $\sup_{t\in [0,T]}\Big(\|\bar\vartheta_{t}-\vartheta_t\|_{C^1_0(\mathbf R^3)^3}+\|\bar R(t)-R(t)\|_{{\rm M}(3)}+\|\bar{\mathbf r}(t)-\mathbf r(t)\|_{\mathbf R^3}\big)<\varepsilon$
where the function $t\in[0,T]\mapsto (R(t),\mathbf r(t))\in{\rm SO}(3)\times\mathbf R^3$ is the unique solution to system \eqref{main_dynamics} with initial data $(R(0),\mathbf r(0))=(\bar R(0),\bar{\mathbf r}(0))$ and control $\vartheta$.
\end{theorem}

This theorem tells us that any 3D microswimmer undergoing approximately any prescribed shape changes can approximately track by swimming any given trajectory.  It may seem surprising that the shape changes, which are supposed to be the control of our problem, can also be somehow preassigned. Actually, the trick is that they can only be {\it approximately} prescribed. We are going to show that arbitrarily small superimposed shape changes suffice for controlling the swimming motion. 

When no macro shape changes are preassigned we have:
\begin{theorem}(Freestyle Swimming)
\label{main_theorem_free}
Assume that the following data are given: (i) A  function $\bar\vartheta\in D^1_0(\mathbf R^3)$ such that $\int_\Sigma\bar\varTheta\,{\rm d}\sigma =0$ (the reference shape at rest) (ii) A continuous function $t\in[0,T]\mapsto (\bar R(t),\bar{\mathbf r}(t))\in{\rm SO}(3)\times\mathbf R^3$ (the reference trajectory). 
Then, for any $\varepsilon>0$ there exists a function $\vartheta \in D^1_0(\mathbf R^3)$ (the actual shape at rest) such that (i) $\int_\Sigma \varTheta\,{\rm d}\sigma=\mathbf 0$ (ii) $\|\bar\vartheta-\vartheta\|_{D^1_0(\mathbf R^3)}<\varepsilon$ and (iii) for almost any $4$-uplet $(\mathbf V_1,\ldots,\mathbf V_4)\in (C^1_0(\mathbf R^3)^3)^4$ satisfying $\int_\Sigma \mathbf V_i\,{\rm d}x=\mathbf 0$, $\int_\Sigma \varTheta\times\mathbf V_i\,{\rm d}\sigma=\mathbf 0$ and $\int_\Sigma \mathbf V_i\times\mathbf V_j\,{\rm d}\sigma=\mathbf 0$ ($i,j=1,\ldots,4$), there exists a function  $t\in[0,T]\mapsto s(t):=(s_1(t),\ldots,s_4(t))^\ast\in\mathbf R^4$ (which can be chosen analytic) such that, using  $\vartheta_t:=\vartheta+\sum_{i=1}^4s_i(t)\mathbf V_i\in D^1_0(\mathbf R^3)$ as control in the dynamics \eqref{main_dynamics}, we get $\sup_{t\in [0,T]}\big(\|\bar R(t)-R(t)\|_{{\rm M}(3)}+\|\bar{\mathbf r}(t)-\mathbf r(t)\|_{\mathbf R^3}\big)<\varepsilon$
where the function $t\in[0,T]\mapsto (R(t),\mathbf r(t))\in{\rm SO}(3)\times\mathbf R^3$ is the unique solution to ODEs \eqref{main_dynamics} with initial data $(R(0),\mathbf r(0))=(\bar R(0),\bar{\mathbf r}(0))$.
\end{theorem}

We claim in this Theorem that any 3D microswimmer (maybe up to an arbitrarily small modification of its initial shape) is able to swim by means of allowable deformations (i.e. satisfying the constraints \eqref{self-propelled-cond}) obtained as a suitable combination of  almost any given four basic movements. 

If we still seek the control function $\vartheta_t$ as a combination of a finite number of elementary deformations, i.e. in the form
\begin{equation}
\label{finite_dim}
\vartheta_t=\vartheta+\sum_{i=1}^ns_i(t)\mathbf V_i,
\end{equation}
 where $t\in[0,T]\mapsto s(t):=(s_1(t),\ldots,s_n(t))^\ast\in\mathbf R^n$ is in $L^1([0,T],\mathbf R^n)$, $\int_\Sigma{\varTheta}\,{\rm d}x=\mathbf 0$ and $(\mathbf V_1,\ldots,\mathbf V_n)\in(C^1_0(\mathbf R^3)^3)^n$ is a fixed family of $n$ vector fields satisfying $\int_\Sigma\mathbf V_i\,{\rm d}x=\mathbf 0$, $\int_\Sigma{\varTheta}\times\mathbf V_i\,{\rm d}x=\mathbf 0$ and $\int_\Sigma\mathbf V_i\times\mathbf V_j\,{\rm d}x=\mathbf 0$ ($i,j=1,\ldots,n$) we can state the following result:
\begin{theorem}(Existence of an optimal control)
\label{theo:optimal}
Let $f:{\rm SO}(3)\times\mathbf R^3\times D^1_0(\mathbf R^3)\times (C^1_0(\mathbf R^3)^3)\to \mathbf R$ be a continuous function, convex in the third variable and let $\mathcal K$ be a compact of $\mathbf R^n$. Let $(R_0,\mathbf r_0,\vartheta_0)$ and $(R_1,\mathbf r_1,\vartheta_1)$ be two elements of ${\rm SO}(3)\times\mathbf R^3\times D^1_0(\mathbf R^3)$ such that there exists a control function $\vartheta_t$ (i) having the form \eqref{finite_dim} with $s(t)\in\mathcal K$ for a.e. $t\in[0,T]$, (ii) satisfying $\vartheta_{t=0}=\vartheta_0$, $\vartheta_{t=T}=\vartheta_1$ and (iii) steering  the dynamics \eqref{main_dynamics} from $(R_0,\mathbf r_0)$ (at $t=0$) to $(R_1,\mathbf r_1)$ (at $t=T$). Then, among all of the control functions satisfying (i-iii), there exists an optimal control $\vartheta^\star_t$ realizing the minimum of the cost
$$\int_0^T f(R(t),\mathbf r(t),\vartheta_t,\partial_t\vartheta_t)\,{\rm d}t.$$
\end{theorem}

The proofs of  these results rely on the following leading ideas: First, we shall identify a set of parameters necessary to thoroughly characterize a swimmer and its way of swimming (these parameters are its shape and a finite number of basic movements, satisfying the constraints \eqref{self-propelled-cond}). Any set of such parameters will be termed a {\it swimmer signature} (denoted SS in short). Then, the set of all of the SS will be shown to be an (infinite dimensional) analytic connected  embedded submanifold of a  Banach space. 

The second step of the reasoning will consist in proving that the swimmer's ability to track any given trajectory (while undergoing approximately any preassigned shape changes) is related to the vanishing of some analytic functions depending on the SS. These functions are connected to the determinant of some vector fields and their Lie brackets (we will invoke classical results of Geometric Control Theory). Eventually, by direct calculation, we will prove that at least one swimmer (corresponding to one particular SS) has this ability. An elementary property of analytic functions will eventually allow us to conclude that almost any SS (or equivalently any microswimmer) has this property.

Eventually, the existence of an optimal control in Theorem~\ref{theo:optimal} is a straightforward consequence of Filippov Theorem (see \cite[Chap. 10]{Agrachev:2004aa})

\subsection{Outline of the paper}
The next Section is dedicated to the notion of swimmer signature (definition and properties). In Section~\ref{sensiti:mass} we show that the matrix $\mathbb M(t)$ and the vector $\mathbf N(t)$ (in \eqref{dynamics}) are analytic functions in the SS (swimmer signature, seen as a variable) and in Section~\ref{SEC:control_problem} we will restate the control problem in order to fit with the general framework of Geometric Control Theory. In the same Section, a particular case of swimmer will be shown to be controllable. In Section~\ref{sec:main_results} the proof of the main results will be carried out. Section~\ref{sec:conclusion} contains some words of conclusion. Technical results and definitions are gathered in the appendix in order to make the paper more readable.  
\section{Swimmer Signature}
\label{section:SS}
A {\it swimmer signature} is a set of parameters characterizing swimmers whose deformations consist in a combination of a finite number of basic movements. 
\begin{definition}
\label{def:sc}
For any positive integer $n$, we denote $\mathcal C(n)$ the subset of $D^1_0(\mathbf R^3)\times (C^1_0(\mathbf R^3)^3)^n$ consisting of all of the pairs $c:=(\vartheta,\mathcal V)$ such that, 
denoting $\varTheta:={\rm Id}+\vartheta$ and $\mathcal V:=(\mathbf V_1,\ldots,\mathbf V_n)$, the following conditions hold (i) the set $\{\mathbf V_i|_{\Sigma}\cdot\mathbf e_k,\,1\leq i\leq n,\,k=1,2,3\}$ is a free family in $C^1(\Sigma)$ (ii) every pair $(\mathbf V,\mathbf V')$ of elements of $\{\varTheta,\mathbf V_1,\ldots,\mathbf V_n\}$ satisfies
$\int_{\Sigma}\mathbf V\,{\rm d}x=\mathbf 0$ and $\int_{\Sigma}\mathbf V\times \mathbf V'\,{\rm d}x=\mathbf 0$.

We call swimmer signature (SS in short) any element $c$ of $\mathcal C(n)$.
\end{definition}

By definition, $D^1_0(\mathbf R^3)$ is open in $C^1_0(\mathbf R^3)^3$ (see appendix, Section~\ref{SEC:diffeo}). We deduce that for any $c\in\mathcal C(n)$, the set $\{s:=(s_1,\ldots,s_n)^\ast\in\mathbf R^n\,:\,\vartheta+\sum_{i=1}^n s_i\mathbf V_i\in D^1_0(\mathbf R^3)\}$ is open as well in $\mathbf R^n$ and we denote $\mathcal S(c)$ its connected component containing $s=0$.
\begin{definition}
For any positive integer $n$, we call {\it swimmer full signature} (SFS in short) any pair $\mathbf c:=(c,s)$ such that $c\in\mathcal C(n)$ and $s\in\mathcal S(c)$. We denote $\mathcal C_F(n)$ the set of all of these pairs.
\end{definition}
\subsubsection*{Restatement of the problem in terms of swimmer signature (SS) and swimmer full signature (SFS)}
Pick a SS, $c=(\mathcal\vartheta,\mathcal V)\in \mathcal C(n)$ with $\mathcal V:=(\mathbf V_1,\ldots,\mathbf V_n)$ (for some integer $n$). Denote $\varTheta:={\rm Id}+\vartheta$ and for all $s\in\mathcal S(c)$, $\varTheta_s:={\rm Id}+\vartheta+\sum_{i=1}^n s_i\mathbf V_i$ ($\mathbf c:=(c,s)\in \mathcal C(n)$ is hence a SFS). The body of the swimmer occupies the domain $\bar{\mathcal B} :=\varTheta(\bar B)$ at rest and $\bar{\mathcal B}_{\mathbf c}:=\varTheta_s(\bar B)$ (for any $s\in\mathcal S(c)$) when swimming.
Notice that within this construction, the shape changes on a time interval $[0,T]$ ($T>0$) are merely given through an absolutely continuous  function $t:[0,T]\mapsto s(t)\in\mathcal S(c)$. 
If $t\in[0,T]\mapsto \dot s(t)\in\mathbf R^n$ stands for its time derivative in $L^1([0,T],\mathbf R^n)$, the Lagrangian velocity at a point $x$ of $\bar B$ is $\sum_{i=1}^n\dot s_i(t) \mathbf V_i(x)$ while the Eulerian velocity at a point $x\in\bar{\mathcal B}_{\mathbf c}$ is $\sum_{i=1}^n\dot s_i(t) \mathbf w^i_s(x)$ with $\mathbf w^i_s(x):=\mathbf V_i(\varTheta_s^{-1}(x))$. Due to assumption~(ii) of Definition~\ref{def:sc}, the constraints \eqref{self-propelled-cond} are automatically satisfied.

The elementary fluid velocities and elementary pressure functions corresponding to the rigid motions depend only on the SFS. Therefore, they will be denoted in the sequel $\mathbf u_i(\mathbf c)$ and $p_i(\mathbf c)$ to emphasize this dependence. The same remark holds for the matrix $\mathbb M(t)$ whose notation is turned into $\mathbb M(\mathbf c)$. The elementary velocity and pressure $(\mathbf u_d, p_d)$ connected to the shape changes can be decomposed into $\mathbf u_d=\sum_{i=1}^n\dot s_i\mathbf w_i({\mathbf c})$ and $p_d=\sum_{i=1}^n\dot s_i \pi_i({\mathbf c})$ respectively. In this sum, each pair $(\mathbf w_i(\mathbf c),\pi_i(\mathbf c))$ solves the Stokes equations in $\mathcal F_{\mathbf c}:=\mathbf R^3\setminus \bar{\mathcal B}_{\mathbf c}$ with boundary conditions $\mathbf w_i(\mathbf c)=\mathbf w^i_s$ on $\Sigma_{\mathbf c}:=\partial {\mathcal B}_{\mathbf c}$.

Introducing the matrix $\mathbb N(\mathbf c)$, whose elements are 
$$N_{ij}(\mathbf c):=\begin{cases}\int_{\Sigma_{\mathbf c}}(x\times \mathbf
e_i)\cdot\mathbb T(\mathbf w_j(\mathbf c),\pi_j(\mathbf c))\mathbf n{\rm d}\sigma&(1\leq i\leq 3,\,1\leq j\leq n);\\
\int_{\Sigma_{\mathbf c}}\mathbf
e_{i-3}\cdot\mathbb T(\mathbf w_j(\mathbf c),\pi_j(\mathbf c))\mathbf n{\rm d}\sigma&(1\leq i\leq 6,\,1\leq j\leq n);
\end{cases}$$
(recall that the viscosity $ \mu$ can be chosen equal to 1), the dynamics \eqref{dynamics} can now be rewritten in the form:
\begin{equation}
\label{dynamics:1}
\begin{pmatrix}\boldsymbol\Omega\\
\mathbf v\end{pmatrix}=-\mathbb M(\mathbf c)^{-1}\mathbb N(\mathbf c)\dot s,\qquad (0<t< T).
\end{equation}
Let us focus on the properties of $\mathcal C(n)$ and $\mathcal C_F(n)$.

\begin{theorem}
\label{theorem:1}
For any positive integer $n$, the set  ${\mathcal C}(n)$ is an analytic connected embedded submanifold  of $C^1_0(\mathbf R^3)^3\times (C^1_0(\mathbf R^3)^3)^n$ of codimension $N:=3(n+2)(n+1)/2$.
\end{theorem}

The definition and the main properties of Banach space valued analytic functions are summarized in the article \cite{Whittlesey:1965aa}. 

\begin{proof}
For any $c:=(\vartheta,\mathcal V)\in C^1_0(\mathbf R^3)^3\times (C^1_0(\mathbf R^3)^3)^n$, denote $\mathbf V_0:={\rm Id}+\vartheta$ and $\mathcal V:=(\mathbf V_1,\ldots,\mathbf V_n)$. Then, define 
for $k=0,1,\ldots,n$, the functions $\Lambda_k: C^1_0(\mathbf R^3)^3\times (C^1_0(\mathbf R^3)^3)^n\to\mathbf R^{3(n+1-k)} $ by $\Lambda_k(c):=
\Big(\int_{\Sigma}\mathbf V_k\,{\rm d}x,\,
\int_{\Sigma}\mathbf V_k\times \mathbf V_{k+1}\,{\rm d}x,\,\ldots,
\int_{\Sigma}\mathbf V_k\times \mathbf V_n\,{\rm d}x\Big) ^\ast$.
Every function $\Lambda_k$ is analytic and so is $\Lambda:=(\Lambda_0,\ldots,\Lambda_n)^\ast:C^1_0(\mathbf R^3)^3\times (C^1_0(\mathbf R^3)^3)^n\to \mathbf R^N$ ($N:=3(n+2)(n+1)/2$). In order to prove that $\partial_c\Lambda(c)$ (the differential of $\Lambda$ at the point $c$) is onto for any $c\in\mathcal C(n)$, assume that there exist $(n+2)(n+1)/2$ vectors $\boldsymbol\alpha_i^j\in\mathbf R^3$ ($0\leq i\leq j\leq n$) such that:
\begin{equation}
\label{gamma_onto}
\sum_{i=0}^n\boldsymbol\alpha_i\cdot\langle\partial_c\Lambda(c),c^h\rangle=\mathbf 0,\qquad\forall\,c^h\in C^1_0(\mathbf R^3)^3\times (C^1_0(\mathbf R^3))^3,
\end{equation}
where $\boldsymbol\alpha_i:=(\boldsymbol\alpha_i^i,\boldsymbol\alpha_i^{i+1},\ldots,\boldsymbol\alpha_i^n)^\ast\in\mathbf R^{3(n+1-i)}$ ($j=0,\ldots,n$) and $c^h:=(\vartheta^h,\mathcal V^h)\in C^1_0(\mathbf R^3)^3\times (C^1_0(\mathbf R^3)^3)^n$ with $\mathbf V_0^h:={\rm Id}+\vartheta^h$ and $\mathcal V^h:=(\mathbf V^h_1,\ldots,\mathbf V^h_n)$.
Reorganizing the terms in \eqref{gamma_onto}, we obtain that:
$$
\sum_{k=0}^n\int_{\Sigma}\mathbf V_k^h\cdot\Big[\sum_{j=0}^{k-1}\boldsymbol\alpha_j^k\times\mathbf V_j+\boldsymbol\alpha_k^k-\sum_{j=k+1}^{n}\boldsymbol\alpha_k^j\times\mathbf V_j\Big]{\rm d}x=0.
$$
Since this identity has to be satisfied for any $(\vartheta^h,\mathcal V^h)\in C^1_0(\mathbf R^3)^3\times (C^1_0(\mathbf R^3))^3$, we deduce that, for every $k=0,\ldots,n$:
\begin{equation}
\label{identity:1}
\sum_{j=0}^{k-1}\boldsymbol\alpha_j^k\times\mathbf V_j|_\Sigma+\boldsymbol\alpha_k^k-\!\!\!\sum_{j=k+1}^{n}\boldsymbol\alpha_k^j\times\mathbf V_j|_\Sigma=\mathbf 0.\end{equation}
Integrating this equality over $\Sigma$, we get that $\boldsymbol\alpha^k_k=\mathbf 0$ ($k=0,\ldots,n$).  Taking into account Hypothesis (ii) of Definition~\ref{def:sc}, the identity \eqref{identity:1} with $k=0$ leads to $\boldsymbol\alpha_0^j=\mathbf 0$ for every $j=1,\ldots,n$. There are no more terms involving $\mathbf V_0$ in the other equations and invoking again Hypothesis (ii) we eventually get $\boldsymbol\alpha_i^j=\mathbf 0$ for $1\leq i<j\leq n$. So, equality \eqref{gamma_onto} entails that $\boldsymbol\alpha_i=\mathbf 0$ for all $i=0,\ldots,n$ and the mapping $\partial_c\Lambda(c)$ is indeed onto for all $c\in\mathcal C(n)$.

The linear space $X=\mathrm{Ker}\,\partial_c \Lambda(c)$ is closed since $\Lambda$ is analytic. Let $Y$ be an algebraic supplement of $X$ in $C^1_0(\mathbf R^3)^3 \times (C^1_0(\mathbf R^3)^3)^n$, and denote by $P_Y$ the linear projection onto $Y$ along $X$. A crucial observation is that the linear space $Y$ is isomorphic to $\mathbf{R}^N$ and hence it is finite dimensional and closed in  $C^1_0(\mathbf R^3)^3 \times (C^1_0(\mathbf R^3)^3)^n$. Define the analytic mapping $f:X\times Y \rightarrow \mathbf{R}^N$ by
$f(x,y)=\Lambda(c+x+y)$. The mapping $\partial_yf(0,0)=\partial_c\Lambda(c)\circ P_Y$ being onto, the implicit function theorem (analytic version in Banach spaces, see \cite{Whittlesey:1965aa}) asserts that there exist an open neighborhood $\mathcal O_1$ of $0$ in $X$, an open neighborhood $\mathcal O_2$ of $0$ in $Y$, and an analytic mapping $g:\mathcal O_1 \rightarrow Y$ such that $g(0)=0$ and, for every $(x,y)$ in $\mathcal O_1\times \mathcal O_2$, the two following assertions are equivalent: (i) $f(x,y)=0$ (or, in other words, $c+x+y$ belongs to $\mathcal{C}(n)$), and
(ii) $y=g(x)$.
The analytic mapping $g$ provides a local parameterization of $\mathcal{C}(n)$ in a neighborhood of $c$.

In order to prove that $\mathcal C(n)$ is path-connected, consider two elements $c^\dagger:=(\vartheta^\dagger,\mathcal V^\dagger)$ and $c^\ddagger:=(\vartheta^\ddagger,\mathcal V^\ddagger)$ of $\mathcal C(n)$ and denote $\varTheta^\dag:={\rm Id}+\vartheta^\dag$, $\mathcal V^\dag:=(\mathbf V_1^\dag,\ldots,\mathbf V_n^\dag)$ and $\varTheta^\ddag:={\rm Id}+\vartheta^\ddag$, $\mathcal V^\ddag:=(\mathbf V_1^\ddag,\ldots,\mathbf V_n^\ddag)$.
According to Definition~\ref{def_D10}, $D^1_0(\mathbf R^3)$ is open and connected. This entails that it is always possible to find a continuous, piecewise linear path $t:[0,1]\mapsto \bar\vartheta_t\in D^1_0(\mathbf R^3)$ such that $\bar\vartheta_{t=0}=\vartheta^\dag$ and $\bar\vartheta_{t=1}=\vartheta^\ddag$. We introduce $0=t_0<t_1<\ldots<t_k=1$, a subdivision of the interval $[0,1]$ such that $t\mapsto\bar\vartheta_t$ is linear on every subinterval $[t_j,t_{j+1}]$ ($j=0,\ldots,k-1$) and we denote $\bar\varTheta_t:={\rm Id}+\bar\vartheta_t$, $\bar\vartheta^j:=\bar\vartheta_{t=t_j}$, $\bar\varTheta^j:={\rm Id}+\bar\vartheta^j$  ($j=0,\ldots,k$).  Since $C^1_0(\mathbf R^3)^3$ is an infinite dimensional Banach space, it is always possible to find by induction $\mathbf W_1,\mathbf W_2,\ldots,\mathbf W_n$ in $C^1_0(\mathbf R^3)^3$ such that (i) both families 
$\{\mathbf W_1|_{\Sigma}\cdot \mathbf e_k,\ldots,\mathbf W_n|_{\Sigma}\cdot \mathbf e_k,\mathbf V_1^\dag|_{\Sigma}\cdot \mathbf e_k,\ldots,\mathbf V_n^\dag|_{\Sigma}\cdot \mathbf e_k,\,k=1,2,3\}$ and $\{\mathbf W_1|_{\Sigma}\cdot \mathbf e_k,\ldots,\mathbf W_n|_{\Sigma}\cdot \mathbf e_k,\mathbf V_1^\ddag|_{\Sigma}\cdot \mathbf e_k,\ldots,\mathbf V_n^\ddag|_{\Sigma}\cdot \mathbf e_k,\,k=1,2,3\}$ are free in $C^1_0(\mathbf R^3)$ and (ii) for any pair of elements $\mathbf V$, $\mathbf V'$, both picked in the same family, $\int_\Sigma\mathbf V{\rm d}x=\mathbf 0$, $\int_\Sigma\bar\varTheta^j\times\mathbf V{\rm d}x=\mathbf 0$ (for all $j=1,\ldots,k$) and $\int_\Sigma\mathbf V\times\mathbf V'{\rm d}x=\mathbf 0$. Define the function $t\in[0,1]\mapsto\mathbf V^i_t\in C^1_0(\mathbf R^3)^3$ by $\mathbf V^i_t:=(1-2t)\mathbf V_i^\dag+2t\mathbf W_i$ if $0\leq t\leq 1/2$ and $\mathbf V^i_t:=(2-2t)\mathbf W_i+(2t-1)\mathbf V^\ddag$ if $1/2<t\leq 1$ and denote $\mathcal V_t:=(\mathbf V^1_t,\ldots,\mathbf V^n_t)\in (C^1_0(\mathbf R^3)^3)^n$. Eventually, a continuous function linking $c^\dag$ to $c^\ddag$ is given by  $t\in[0,1]\mapsto c_t\in\mathcal C(n)$ with $c_t:=(\vartheta^\dag,\mathcal V_{3t/2})$ if $0\leq t\leq 1/3$, $c_t:=(\vartheta_{3t-1},\mathcal V_{1/2})$ if $1/3<t\leq 2/3$ and $c_t:=(\vartheta^\ddag,\mathcal V_{3t/2-1/2})$ if $2/3<t\leq 1$.
\end{proof}

We omit the proof of the following corollary, similar to that of the theorem above:

\begin{corollary}
\label{cor:1}
For any positive integer $n$, the set $\mathcal C_F(n)$ is an analytic connected embedded submanifold of $C^1_0(\mathbf R^3)^3\times (C^1_0(\mathbf R^3)^3)^n\times \mathbf R^n$ of codimension $N:=3(n+2)(n+1)/2$.
\end{corollary}

We denote by $\varPi$ the projection of $\mathcal C(n)$ onto $D^1_0(\mathbf R^3)$ defined by $\varPi(c)=\vartheta$ for all $c:=(\vartheta,\mathcal V)\in\mathcal C(n)$. The proof of the following corollary is a straightforward consequence of arguments already used in the proof of Theorem~\ref{theorem:1}:

\begin{corollary}
\label{cor:2}
For any positive integer $n$ and for any $\vartheta\in\varPi(\mathcal C(n))$, the section $\varPi^{-1}(\{\vartheta\})$ is an embedded connected analytic submanifold of $\{\vartheta\}\times(C^1_0(\mathbf R^3)^3)^n$ (identified with $(C^1_0(\mathbf R^3)^3)^n$) of codimension $3n(n+3)/2$.
\end{corollary}

\section{Sensitivity Analysis of the Matrices $\mathbb M(\mathbf c)$ and $\mathbb N(\mathbf c)$}
\label{sensiti:mass}
For any positive integers $k$ and $l$, we denote ${\rm M}(k,l)$ the vector space of the matrices of size $k\times l$ (or simply ${\rm M}(k)$ when $l=k$).
\begin{theorem}
\label{theorem:2}
For any positive integer $n$, the mappings $\mathbf c\in\mathcal C_F(n)\mapsto \mathbb M(\mathbf c)\in {\rm M}(6)$ and $\mathbf c\in\mathcal C_F(n)\mapsto \mathbb N(\mathbf c)\in {\rm M}(6,n)$ are analytic.
\end{theorem}

Let us begin with a preliminary lemma of which the statement requires introducing some material.
Thus, we denote $F:=\mathbf R^3\setminus \bar B$ (remember that $B$ is the unit ball, $\Sigma:=\partial B$ and $\mathbf n$ is the unit normal to $\Sigma$ directed toward the interior of $B$). For all  $\xi\in D^1_0(\mathbf R^3)$, we set $\varXi:={\rm Id}+\xi$, $\mathcal B_\xi:=\varXi(B)$, $\mathcal F_\xi:=\varXi(F)$ and $\Sigma_\xi:=\varXi(\Sigma)$. We denote $\mathbf q:=(\xi,\mathcal W)$, with $\mathcal W:=(\mathbf W^1,\mathbf W^2)\subset (C^1_0(\mathbf R^3)^3)^2$, the elements of $\mathcal Q:=D^1_0(\mathbf R^3)\times(C^1_0(\mathbf R^3)^3)^2 $ and $\mathbf w^i_\xi:=\mathbf W^i(\varXi^{-1})$ ($i=1,2$). 
Finally, for every $\mathbf q\in\mathcal Q$, we define:
\begin{equation}
\label{def:M}
\varPhi(\mathbf q):=\int_{\mathcal F_\xi}D(\mathbf u^1_{\mathbf q}): D(\mathbf u^2_{\mathbf q})\,{\rm d}x,
\end{equation}
where, for every $i=1,2$, there exists a function $p^i_{\mathbf q}\in L^2(\mathcal F_\xi)$ such that the pair $(\mathbf u^i_{\mathbf q},p^i_{\mathbf q})\in (W^1_0(\mathcal F_\xi))^3\times L^2(\mathcal F_\xi)$ solves 
the Stokes system:
 \begin{subequations}
 \label{stokes_system:1}
\begin{alignat}{3}
-\Delta \mathbf u_{\mathbf q}^i+\nabla p^i_{\mathbf q}&=\mathbf 0&\quad&\text{ in }\mathcal F_\xi,\\
\nabla\cdot \mathbf u^i_{\mathbf q}&=0&&\text{ in }\mathcal F_\xi,\label{stokes_sys22}\\
\mathbf u^i_{\mathbf q}&=\mathbf w^i_\xi&&\text{ on }\Sigma_\xi.\label{stokes_sys23}
\end{alignat}
\end{subequations}
The first equation has to be understood in the weak sense, namely:
\begin{equation}
\label{varia:1}
\int_{\mathcal F_\xi}\!\!\!\nabla \mathbf u^i_{\mathbf q}:\nabla \mathbf v\,{\rm d}x-\int_{\Sigma_\xi}\!\!\!p^i_{\mathbf q}(\nabla\cdot \mathbf v)\,{\rm d}x=\mathbf 0,
\qquad \forall\,\mathbf v\in  (\stackrel{\circ}{W^1_0}(\mathcal F_\xi))^3.
\end{equation}
Recall that the function spaces are defined in the Appendix, Section~\ref{SEC:diffeo}.
\begin{lemma}
\label{lemma:import}
The mapping $\mathbf q\in\mathcal Q\mapsto \varPhi(\mathbf q)\in\mathbf R$ is analytic.
\end{lemma}
\begin{proof}
We pull back equality \eqref{varia:1} onto the domain $F$ using the diffeomorphism $\varXi$. We get:
\begin{subequations}
\begin{equation}
\int_{ F}\!\!\nabla \mathbf U^i_{\mathbf q}\mathbb A_\xi: \nabla \mathbf V\,{\rm d}x-
\int_{F}\!\!P^i_{\mathbf q}\mathbb B_\xi:\nabla \mathbf V\,{\rm d}x=\mathbf 0,\qquad\forall\,\mathbf V\in  (\stackrel{\circ}{W^1_0}(F))^3,
\end{equation}
where $\mathbf U^i_{\mathbf q}:=\mathbf u^i_{\mathbf q}\circ\varXi$, $P^i_{\mathbf q}:=p^i_{\mathbf q}\circ\varXi$, $J_\xi:=\det(\nabla\varXi)$,
$\mathbb A_\xi:=(\nabla\varXi^\ast\nabla\varXi)^{-1}J_\xi$ and $\mathbb B_\xi:=(\nabla\varXi^\ast)^{-1}J_\xi$. Likewise, (\ref{stokes_sys22}-\ref{stokes_sys23}) can be turned into:
\begin{alignat}{3}
\mathbb B_\xi:\nabla \mathbf U^i_{\mathbf q}&=0,&\quad&\text{in }F,\\
\mathbf U^i_{\mathbf q}&=\mathbf W^i&&\text{on }\Sigma.
\end{alignat}
\end{subequations}
 We now claim that the mapping $\xi\in D^1_0(\mathbf R^3)\mapsto \mathbb A_\xi-{\rm Id} \in E^0_0(\mathbf R^3,{\rm M}(3))$ is analytic. Indeed, the mappings $\xi\in D^1_0(\mathbf R^3)\mapsto \nabla\varXi^\ast\nabla\varXi-{\rm Id}\in E^0_0(\mathbf R^3,{\rm M}(3))$, $A\in E^0_0(\mathbf R^3,{\rm M}(3))\mapsto ({\rm Id}+A)^{-1}-{\rm Id}\in E^0_0(\mathbf R^3,{\rm M}(3))$ and $\xi\in D^1_0(\mathbf R^3)\mapsto J_\xi-1\in C^0_0(\mathbf R^3)$  are analytic. Then, for $i=1,2$, we define the analytic  functions $\varGamma^i:\mathcal Q\times (W^1_0(F))^3\times L^2(F)\to (W^{-1}_0(F))^3\times L^2(F)\times (H^{1/2}(\Sigma))^3$ by:
$$
\varGamma^i(\mathbf q,\mathbf U,P):=\begin{pmatrix}
\langle \mathbb A_\xi,\mathbf U,\cdot\rangle-\langle \mathbb B_\xi,P,\cdot\rangle\\
\mathbb B_\xi:\nabla \mathbf U\\
\gamma_{\Sigma}(\mathbf U-\mathbf W^i)
\end{pmatrix},$$
where $\gamma_{\Sigma}:(W^1(F))^3\to(H^{1/2}(\Sigma))^3$ is the trace operator and
\begin{alignat*}{3}
\langle \mathbb A_\xi,\mathbf U,\mathbf V\rangle&:=\int_{\mathcal F}\nabla \mathbf U\mathbb A_\xi: \nabla \mathbf V\,{\rm d}x,&\quad&(\mathbf U\in(W^1(F))^3,\,\mathbf V\in(W^1_0( F))^3),\\
\langle \mathbb B_\xi,P,\mathbf V\rangle&:=\int_{\mathcal F}P\mathbb B_\xi:\nabla \mathbf V\,{\rm d}x,&&(P\in L^2(F),\,\mathbf V\in(W^1_0(F))^3).
\end{alignat*}
We wish now to apply the implicit function theorem (analytic version in Banach spaces, as stated in \cite{Whittlesey:1965aa}) to the analytic function $\varGamma^i$. Observe however that we are only interested in the regularity result. Indeed, according to Proposition~\ref{prop:change_of_variables}, we already know that for all $i=1,2$ and all $\mathbf q\in\mathcal Q$, there exists a unique pair $(U^i_{\mathbf q},P^i_{\mathbf q})\in (W^1_0(F))^3\times L^2(F)$ such that $\varGamma^i(\mathbf q,\mathbf U^i_{\mathbf q},P^i_{\mathbf q})=\mathbf 0$. For every $\mathbf q\in\mathcal Q$,
the partial derivative $\partial_{(U,P)} \varGamma^i(\mathbf q,U^i_{\mathbf q},P^i_{\mathbf q})$ can be readily computed. Indeed, we have:
\begin{equation}
\label{riez}
\langle \partial_{(U,P)} \varGamma^i(\mathbf q,U^i_{\mathbf q},P^i_{\mathbf q}),(\boldsymbol\chi,\pi)\rangle=
\begin{pmatrix}
\langle \mathbb A_\xi,\boldsymbol\chi,\cdot\rangle-\langle \mathbb B_\xi,\pi,\cdot\rangle\\
\mathbb B_\xi:\nabla \boldsymbol\chi\\
\gamma_{\Sigma}(\boldsymbol\chi)
\end{pmatrix},\quad\forall\,(\boldsymbol\chi,\pi)\in (W^1_0(F))^3\times L^2(\Sigma).
\end{equation}
Let $(\mathbf f,\eta,\mathbf g)$ be any element of $(W^{-1}_0(F))^3\times L^2(F)\times (H^{1/2}(F))^3$. The equation 
$\langle \partial_{(U,P)} \varGamma^i(\mathbf q,(U^i_{\mathbf q},P^i_{\mathbf q}),(\boldsymbol\chi,\pi)\rangle=(\mathbf f,\eta,\mathbf g)$,
is equivalent to:
\begin{align*}
\int_{F}\!\nabla \boldsymbol\chi\mathbb A_\xi : \nabla \mathbf V\,{\rm d}x-
\int_{F}\!\pi\,\mathbb B_\xi:\nabla \mathbf V\,{\rm d}x&=\langle \mathbf f,\mathbf V\rangle_{(W^{-1}_0(F))^3\times(\stackrel{\circ}{W^1_0}(F))^3},\quad \forall\,\mathbf V\in (\stackrel{\circ}{W^1_0}(F))^3,\\
\mathbb B_\xi:\nabla\boldsymbol\chi&=\eta,\quad\text{in }F,\\
\boldsymbol\chi&=\mathbf g\quad\text{on }\Sigma.
\end{align*}
According to Proposition~\ref{prop:change_of_variables}, there exists a unique solution $(\boldsymbol\chi,\pi)\in(W^1_0(F))^3\times L^2(F)$ such that $\|\boldsymbol\chi\|_{(W^1_0(F))^3}+\|\pi\|_{L^2(F)}\leq C_\xi\big[\|\mathbf f\|_{(W^{-1}_0)^3}+\|\eta\|_{L^2(F)}+\|\mathbf g\|_{(H^{1/2}(\Sigma))^3}\big]$ where the constant $C_\xi>0$ depends on $\xi$ only. We infer that for every $\mathbf q\in\mathcal Q$, $\partial_{(U,P)} \varGamma^i(\mathbf q,U^i_{\mathbf q},P^i_{\mathbf q})$ is a continuous isomorphism from $(W^1_0(F))^3\times L^2(F)$ onto $(W^{-1}_0(F))^3\times L^2(F)\times (H^{1/2}(\Sigma))^3$.
The implicit function theorem applies and asserts that the mappings $\mathbf q\in\mathcal Q\mapsto (\mathbf U^i_{\mathbf q},P^i_{\mathbf q})\in (W^1_0(F))^3\times L^2(F)$ ($i=1,2$) are analytic.

To conclude the proof, it remains only to observe that the function $\varPhi(\mathbf q)$ introduced in \eqref{def:M} can be rewritten, upon a change of variables as $$\varPhi(\mathbf q)=\frac{1}{4}\int_{F}(\nabla U_{\mathbf q}^1\nabla\varXi^{-1}+(\nabla U_{\mathbf q}^1\nabla\varXi^{-1})^\ast):(\nabla U_{\mathbf q}^2\nabla\varXi^{-1}+(\nabla U_{\mathbf q}^2\nabla\varXi^{-1})^\ast)J_\xi\,{\rm d}x,$$ 
which is analytic as a composition of analytic functions.
\end{proof}

We can now give the proof of Theorem~\ref{theorem:2}.
\begin{proof}
For any $\mathbf c:=(c,s)\in\mathcal C_F(n)$, where $c:=(\vartheta,\mathcal V)$, we apply the lemma with $\xi:=\vartheta+\sum_{i=1}^ns_i\mathbf V_i$ and $\mathbf W^1,\mathbf W^2\in\{\mathbf e_i\times\varXi,\, \mathbf e_i,\,i=1,2,3\}$ to get that the mapping $\mathbf c\in\mathcal C_F(n)\mapsto \mathbb M(\mathbf c)\in {\rm M}(6)$ is analytic. 
To prove the analyticity of the elements of $\mathbb N(\mathbf c)$, we apply the lemma again with $\xi:=\vartheta+\sum_{i=1}^ns_i\mathbf V_i$, $\mathbf W^1\in \{ \mathbf e_i\times\varXi,\, \mathbf e_i,\,i=1,2,3\}$ and $\mathbf W^2\in\{\mathbf V_1,\ldots,\mathbf V_n\}$. 
\end{proof}
\section{Control Problem}

\label{SEC:control_problem}
\subsection{Controllable swimmer signature}
Let us fix $c\in{\mathcal C}(n)$ (for some positive integer $n$) and recall that $\mathcal S(c)$ is the connected open subspace of $\mathbf R^n$ such that $(c,s)\in\mathcal C_F(n)$.  Introducing $(\mathbf f_1,\ldots,\mathbf f_n)$ an ordered orthonormal basis of $\mathbf R^n$, we can seek the function $t\in[0,T]\mapsto s(t)\in\mathcal S(c)$ as the solution of the ODE $\dot s(t)=\sum_{i=1}^n\lambda_i(t)\mathbf f_i$ where the functions $\lambda_i:t\in[0,T]\mapsto \lambda_i(t)\in\mathbf R$ are the new controls, and rewrite once more the dynamics \eqref{dynamics:1} as:
\begin{equation}
\label{dynamics:2}
\begin{pmatrix}\boldsymbol\Omega\\
\mathbf v\\
\dot s\end{pmatrix}=\sum_{i=1}^n\lambda_i(t)
\begin{pmatrix}-\mathbb M(c,s)^{-1}\mathbb N(c,s)\mathbf f_i\\
\mathbf f_i\end{pmatrix}, \qquad (0<t<T).
\end{equation}
It is worth remarking that in this form, $s$ is no more the control but a state variable and $c\in\mathcal C(n)$ is a parameter of the dynamics. 
Considering \eqref{dynamics:2}, we are quite naturally led to introduce, for all $\mathbf c\in\mathcal C_F(n)$, the vector fields
$\mathbf X_i(\mathbf c):=-\mathbb M(\mathbf c)^{-1}\mathbb N(\mathbf c)\mathbf f_i\in\mathbf R^6$, $\mathbf Y_i(\mathbf c):=(\hat{\mathbf X}^1_i(\mathbf c)
,
{\mathbf X}^2_i(\mathbf c),
\mathbf f_i)^\ast\in T_{\rm Id}{\rm SO}(3)\times\mathbf R^3\times \mathbf R^n$ (we have used here the notation $\mathbf X_i:=(\mathbf X_i^1,\mathbf X_i^2)^\ast\in\mathbf R^3\times\mathbf R^3$) and 
$\mathbf Z_c^i(R,s):=\mathcal R_R\mathbf Y_i(\mathbf c)\in T_R{\rm SO}(3)\times\mathbf R^3\times \mathbf R^n$
where $\mathcal R_R:={\rm diag}(R,R,{\rm Id})\in{\rm SO}(6+n)$ is a bloc diagonal matrix.
The dynamics \eqref{dynamics:2} and the ODE \eqref{complement} can be gathered into a unique differential system:
\begin{equation}
\label{dynamics:3}
\frac{d}{dt}\begin{pmatrix}R\\
\mathbf r\\
s\end{pmatrix}=\sum_{i=1}^n\lambda_i(t)\mathbf Z^i_c(R,s),\qquad(0<t< T).
\end{equation}
For every $i=1,\ldots,n$, the function $(R,\mathbf r,s)\in{\rm SO}(3)\times\mathbf R^3\times\mathcal S(c)\mapsto \mathbf Z^i_c(R,s)\in T_R{\rm SO}(3)\times\mathbf R^3\times \mathbf R^n$ can be seen as an analytic vector field (constant in $\mathbf r$) on the analytic connected manifold $\mathcal M(c):={\rm SO}(3)\times\mathbf R^3\times\mathcal S(c)$. We denote $\zeta$ any element $(R,\mathbf r,s)\in\mathcal M(c)$ and we define $\mathcal Z(c)$ as the family of vector fields $(\mathbf Z^i_c)_{1\leq i\leq n}$ on $\mathcal M(c)$.  
\begin{lemma}Let $c$ be a SS fixed in $\mathcal C(n)$ ($n$ a positive integer). 
If there exists $\zeta\in\mathcal M(c)$ such that ${\rm dim}\,{\rm Lie}_{\zeta}\mathcal Z(c)=6+n$, then the orbit of $\mathcal Z(c)$ through any $\zeta\in\mathcal M(c)$ is equal to the whole manifold $\mathcal M(c)$.
\end{lemma}

\begin{proof}
Rashevsky Chow Theorem (see \cite{Agrachev:2004aa}) applies: If ${\rm Lie}_{\zeta}\mathcal Z(c)=T_{\zeta}\mathcal M(c)$ for all $\zeta\in\mathcal M(c)$ (or more precisely, for all $(R,s)\in{\rm SO}(3)\times\mathcal S(c)$ since $\mathbf Z^i_c$ does not depend on $\mathbf r$) then the orbit of $\mathcal Z(c)$ through any point of $\mathcal M(c)$ is equal to the whole manifold. 
Let us compute the Lie bracket $[\mathbf Z^i_c(R,s),\, \mathbf Z^j_c(R,s)]$ for $1\leq i,j\leq n$ and $(R,s)\in{\rm SO}(3)\times\mathcal S(c)$. We get:
\begin{equation}
\label{identity:lie_brackets}
[\mathbf Z^i_c(R,s),\, \mathbf Z^j_c(R,s)]=\mathcal R_R\begin{pmatrix}
\widehat{(\mathbf X_i^1\times\mathbf X_j^1)}(\mathbf c)\\
({\mathbf X}^1_i\times\mathbf X_j^2-{\mathbf X}^1_j\times\mathbf X_i^2)(\mathbf c)\\
\mathbf 0
\end{pmatrix}
+\mathcal R_R\begin{pmatrix}
\widehat{(\partial_{s_i}{\mathbf X}^1_j-\partial_{s_j}{\mathbf X}_i^1)}(\mathbf c)\\
(\partial_{s_i}{\mathbf X}^2_j-\partial_{s_j}{\mathbf X}_i^2)(\mathbf c)\\
\mathbf 0
\end{pmatrix}.
\end{equation}
By induction, we can similarly prove that the Lie brackets of any order at any point $\zeta\in\mathcal M(c)$ have the same general form, namely the matrix $\mathcal R_R$ multiplied by an element of $T_{({\rm Id},\mathbf 0,s)}\mathcal M(c)$. We deduce that the dimension of the Lie algebra at any point of $\mathcal M(c)$ depends only on $s$. According to the Orbit Theorem (see \cite{Agrachev:2004aa}), the dimension of the Lie algebra is constant along any orbit. But according to the particular form of the vector fields $\mathbf Z_c^i$ (whose last $n$ components form a basis of $\mathbf R^n$), the projection of any orbit on $\mathcal S(c)$ turns out to be the whole set $\mathcal S(c)$ (or, in other words, for any $s\in\mathcal S(c)$ and for any orbit, there is a point of the orbit for which the last component is $s$). Assume now that ${\rm dim}\,{\rm Lie}_{\zeta^\ast}\mathcal Z(c)=6+n$ at some particular point $\zeta^\ast:=(R^\ast,\mathbf r^\ast,s^\ast)\in\mathcal M(c)$. Then, according to the Orbit Theorem, for any $s\in\mathcal S(c)$, there exists at least one point $(R_s,\mathbf r_s,s)\in\mathcal M(c)$ such that  ${\rm dim}\,{\rm Lie}_{(R_s,\mathbf r_s,s)}\mathcal Z(c)=6+n$. But since the dimension of the Lie algebra does not depend on the variables $R$ and $\mathbf r$, we conclude that ${\rm dim}\,{\rm Lie}_{\zeta}\mathcal Z(c)=6+n$ for all $\zeta\in\mathcal M(c)$.
\end{proof}

\begin{definition}
\label{def:cont:SS}
We say that $c$, a SS in $\mathcal C(n)$ (for some integer $n$) is controllable if there exists $\zeta\in\mathcal M(c)$ such that ${\rm dim}\,{\rm Lie}_{\zeta}\mathcal Z(c)=6+n$.
\end{definition}

It is obvious that for a SS to be controllable, the integer $n$ has to be larger or equal to 2. 
The following result is quite classical (a proof can be found in \cite{Chambrion:2011aa}):
\begin{proposition}
\label{prop:prop12}
Let $c\in \mathcal C(n)$ (for some integer $n$) be controllable (with the usual notation $c:=(\vartheta,\mathcal V)$, $\mathcal V:=(\mathbf V_1,\ldots,\mathbf V_n)$ and $\vartheta_s:=\vartheta+\sum_{i=1}^ns_i\mathbf V_i$ for every $s\in\mathcal S(c)$). Then for any given continuous function $t\in[0,T]\mapsto (\bar R(t),\bar{\mathbf r}(t),\bar s(t))\in {\rm SO}(3)\times\mathbf R^3\times \mathcal S(c)$ and for any $\varepsilon>0$, there exist $n$ $C^1$ functions $\lambda_i:[0,T]\to\mathbf R$ ($i=1,\ldots,n$) such that:
\begin{enumerate}
\item $\sup_{t\in[0,T]}\Big(\| \bar R(t)-R(t)\|_{{\rm M}(3)}+\|\bar{\mathbf r}(t)-\mathbf r(t)\|_{\mathbf R^3}+\|\vartheta_{\bar s(t)}-\vartheta_{s(t)}\|_{C^1_0(\mathbf R^3)^3}\Big)<\varepsilon;$
\item $R(T)=\bar R(T)$, $\mathbf r(T)=\bar{\mathbf r}(T)$ and $s(T)=\bar s(T)$; 
\end{enumerate}
where $t\in[0,T]\mapsto(R(t),\mathbf r(t),s(t))\in\mathcal M(c)$ is the unique solution to the ODE \eqref{dynamics:3} with Cauchy data $R(0)=\bar R(0)\in{\rm SO}(3)$, $\mathbf r(0)=\bar{\mathbf r}(0)\in\mathbf R^3$, $s(0)=\bar s(0)\in\mathcal S(c)$.
\end{proposition}

Let us mention some other quite elementary properties that will be used later on:

\begin{proposition} 
\label{properties_of_SS}
\begin{enumerate}
\item If $c:=(\vartheta,\mathcal V)\in\mathcal C(n)$ ($n\geq 2)$ is a controllable SS with $\mathcal V:=(\mathbf V_1,\ldots,\mathbf V_n)\in(C^1_0(\mathbf R^3)^3)^n$ then any $c^+:=(\vartheta,\mathcal V^+)\in\mathcal C(n+1)$  such that $\mathcal V^+:=(\mathbf V_1,\ldots,\mathbf V_n,\mathbf V_{n+1})\in (C^1_0(\mathbf R^3)^3) ^{n+1}$ (for some $\mathbf V_{n+1}\in C^m_0(\mathbf R^3)^3$) is a controllable SS as well.
\item If $c:=(\vartheta,\mathcal V)\in\mathcal C(n)$ ($n\geq 2)$ is a controllable SS, then for any $\vartheta^\star\in\{\vartheta+\sum_{i=1}^ns_i\mathbf V_i,\,s\in\mathcal S(c)\}$ the element  $c^\star:=(\vartheta^\star,\mathcal V)$ belongs to  $\mathcal C(n)$ and is a controllable SS as well.
\item If $c:=(\vartheta,\mathcal V)\in\mathcal C(n)$ ($n\geq 2)$ is a controllable SS, then all of the controllable SS  in the section $\varPi^{-1}(\{\vartheta\})$  form an open dense subset of $\varPi^{-1}(\{\vartheta\})$ (for the induced topology).
\item If there exists a SS in $\mathcal C(n)$ for some $n\geq 2$ then, for any $k\geq n$, all of the controllable SS in $\mathcal C(k)$ form an open dense subset of $\mathcal C(k)$ (for the induced topology).
\end{enumerate}
\end{proposition}

\begin{proof}
The two first assertions are obvious so let us address directly the third point. Denote $\mathcal E_k$ ($k$ positive integer) the set of all of the vectors fields on $\mathcal M(c)$ obtained as Lie brackets of order lower or equal to $k$ from elements of $\mathcal Z(c)$. Then, consider the determinants of all of the different families of $6+n$ elements of $\mathcal E_k$ as analytic functions in the variable $\mathcal V$ (the other variables $\vartheta$ and $s=0$ being fixed). Since $c$ is controllable, there exist at least one $k$ and one family of $6+n$ elements in $\mathcal E_k$ whose determinant is nonzero. According to Corollary~\ref{cor:2} and basic properties of analytic functions (see \cite{Whittlesey:1965aa}), the determinant can vanish only in a closed subset with empty interior of the section $\varPi^{-1}(\{\vartheta\})$ (for the induced topology). The proof of the last point is similar.
\end{proof}
\subsection{Building a controllable swimmer signature}
In this subsection, we are interested in computing the Lie brackets of first order $[\mathbf Z_c^i(R,s),\,\mathbf Z_c^j(R,s)]$ at $(R,s)=({\rm Id},0)$, for a particular SS $c:=({\rm Id},\mathcal V)\in \mathcal C(4)$ (so the shape of the swimmer at rest is the unit ball). We make use of the usual notation $\mathcal V:=(\mathbf V_1,\ldots,\mathbf V_4)$ (to be specified latter on), $s=(s_1,\ldots,s_4)\in\mathcal S(c)$ and $\mathbf c:=(c,s)$. 

To carry out the aforementioned task, we introduce the classical spherical coordinates $(\varrho,\alpha,\beta)$ such that, for all $x:=(x_1,x_2,x_3)^\ast\in\mathbf R^3$, $x\neq 0$, we have $x_1=\varrho\cos(\alpha)\sin(\beta)$, $x_2=\varrho\sin(\alpha)\sin(\beta)$ and $x_3=\varrho\cos(\beta)$.
At each point $(\varrho,\alpha,\beta)$ we define the related local frame $(\mathbf e_\varrho,\mathbf e_\alpha,\mathbf e_\beta)$. For any $n\geq 1$, we call rigid spherical harmonics of degree $-(n+1)$ any function having the form:
\begin{equation}
\label{rigid_spherical}
(\varrho,\alpha,\beta)\mapsto\varrho^{-(n+1)}\sum_{m=-n}^n\gamma_mY_{n,m}(\cos\beta,\alpha),
\end{equation}
where $\gamma_{-n},\ldots,\gamma_n\in\mathbf{R}$ and $Y_{n,m}$ are the classical spherical harmonics of degree $n\in\mathbf N$ and order $m\in\{-n,\hdots,n\}$.

According to Lamb, \cite{Lamb:1993aa} (one can also see the book of Happel and Brenner, \cite[ch. 3.2, p. 62]{Happel_Brenner}), the solution $(\mathbf u,p)$  of the Stokes equations around an immersed body of any shape can be decomposed as follows (in the body frame):
\begin{subequations}
  \label{brenner_form}
  \begin{align}
    \nonumber
    \mathbf u&=\sum_{n=1}^{+\infty} \left(\nabla\times(\chi_{-(n+1)}\varrho\mathbf{e}_\varrho)+\nabla\phi_{-(n+1)}-\frac{n-2}{2n(2n-1)}\varrho^2\nabla p_{-(n+1)}\right.\\
     \label{vBrenner}
    &\hspace{9cm}\left.+ \frac{n+1}{n(2n-1)}p_{-(n+1)}\varrho\mathbf{e}_\varrho\right),\\
    \label{pBrenner}
    p&=\sum_{n=1}^{+\infty} p_{-(n+1)},
  \end{align}
\end{subequations}
where $p_{-(n+1)}$, $\chi_{-(n+1)}$ and $\phi_{-(n+1)}$ are rigid spherical harmonics of degree $-(n+1)$. 

The functions $p_{-(n+1)}$, $\chi_{-(n+1)}$ and $\phi_{-(n+1)}$ (or more precisely the coefficients $\gamma_k$, $k\in\{-n,\ldots,n\}$ arising in  \eqref{rigid_spherical}) have to be determined in order to satisfy the boundary conditions on the surface on the body. This can be done following a method given in \cite{Brenner:1964aa} to which we refer for further details. 

The main interest of writing the solution of the Stokes equations in the form \eqref{brenner_form} is that the entries of the matrix $\mathbb M(\mathbf c)$ and $\mathbb N(\mathbf c)$ can be easily determined.
\begin{lemma}\label{lemma:FT}
  Let $\mathcal S$ be any smooth open bounded domain of $\mathbb R^3$ and denote $\mathcal F:=\mathbf{R}^3\setminus\overline{\mathcal S}$. Let $(\mathbf u,p)\in (W^1_0(\mathcal F))^3\times L^2(\mathcal F)$ be a solution to the Stokes equations given by \eqref{brenner_form} satisfying, for some $n_0\in\mathbf N$, $\chi_{-(n+1)}=\phi_{-(n+1)}=p_{-(n+1)}=0$ for all $n>n_0$. 
  
For $i=1,\ldots,6$, let $(\mathbf u_i,p_i)\in (W^1_0(\mathcal F))^3\times L^2(\mathcal F)$ be the solution to the Stokes equations corresponding to the boundary condition $\mathbf u_i(x)=x\times\mathbf e_i$ if $i\in\{1,2,3\}$ and $\mathbf u_i(x)=\mathbf{e}_{i-3}$ if $i\in\{4,5,6\}$ on $\partial \mathcal S$.
  Then we have,
  \begin{equation}\label{FTBrennet}
    2\left(\int_{\mathcal F}D(\mathbf u):D(\mathbf u_i)\, \mathrm dx\right)_{i=1,\hdots,6}=\begin{pmatrix}-8\pi\nabla(\varrho^3\chi_{-2})\\-4\pi\nabla(\varrho^3p_{-2})\end{pmatrix}\, .
  \end{equation}
\end{lemma}
\begin{proof}
  Let $\tilde{\mathbf u}_i$ be the rigid vector field defined by $\tilde{\mathbf u}_i(x)=x\times\mathbf e_i$ if $i\in\{1,2,3\}$ and $\tilde{\mathbf u}_i(x)=\mathbf{e}_{i-3}$ if $i\in\{4,5,6\}$.
  Since $\mathbf u$ and $\mathbf u_i$ are smooth, we have $2\int_{\mathcal F}D(\mathbf u):D(\mathbf u_i)\, \mathrm dx=\int_{\partial \mathcal S}\mathbb T(\mathbf u,p)\mathbf u_i\cdot\mathbf n\, \mathrm{d}\sigma=\int_{\partial \mathcal S}\mathbb T(\mathbf u,p)\tilde{\mathbf u}_i\cdot \mathbf n\, \mathrm{d}\sigma$, where $\mathbf n$ is the normal to $\partial\mathcal S$ oriented towards the interior of $\mathcal S$.
Let $B(0,R)\subset \mathbf R^3$ be a ball centered at $0$ of radius $R>0$ such that $\mathcal S\subset B(0,R)$ and denote $\mathcal F_R:=\mathcal F\cap B(0,R)$. Using the Green formulae and the fact that for every $i\in\{1,\hdots,6\}$ we have $D(\tilde{\mathbf u}_i):=\bigl(\nabla\tilde{\mathbf u}_i+\nabla{\tilde{\mathbf u}_i}^\ast\bigr)/2=\mathbf 0$, we obtain $\int_{\partial \mathcal S}\mathbb T(\mathbf u,p)\tilde{\mathbf u}_i\cdot\mathbf n\, \mathrm{d}\sigma=-\int_{\partial B(0,R)}\mathbb T(\mathbf u,p)\tilde{\mathbf u}_i\cdot \mathbf n\, \mathrm{d}\sigma$, with $\mathbf n$ the normal to $\partial \mathcal F_R$ oriented towards the exterior of $\mathcal F_R$.
Invoking the $L^2$ orthogonality of the spherical harmonics, we get \eqref{FTBrennet}.
\end{proof}

When the body is specialized to be the unit sphere and the boundary conditions for $\mathbf u$ are $\mathbf e_i\times x$ or $\mathbf e_j$ ($i,j=1,2,3$), the entries of the vectors in \eqref{FTBrennet} are the elements of the matrix $\mathbb M(c,0)$ and we get $\mathbb M(c,0)={\rm diag}(8\pi{\rm Id},4\pi{\rm Id})$. 
Similarly, if $\mathbf u=\mathbf V_i$ ($i=1,\ldots,4$) on the surface of the body, the entries of the vectors in \eqref{FTBrennet} turn out to be the elements of the matrix $\mathbb N(c,0)$.
Let now the vector fields $\mathbf V_i$ be defined by $\mathbf V_i(\varrho,\alpha,\beta):={V}_i(\varrho,\alpha,\beta)\mathbf e_\varrho$ for every $i\in\{1,\hdots,4\}$ with
\begin{subequations}
\label{def:vectors}
\begin{alignat}{1}
  V_1(\varrho,\alpha,\beta) = & \displaystyle{\varrho^{-(3+1)}\Re\left(Y_{3,1}\right)}\\
  V_2(\varrho,\alpha,\beta) = & \displaystyle{\varrho^{-(3+1)}\Im\left(Y_{3,1}\right)}\\
  V_3(\varrho,\alpha,\beta) = & \displaystyle{\varrho^{-(3+1)}\Re\left(Y_{3,2}\right)}\\
  V_4(\varrho,\alpha,\beta) = & \displaystyle{\varrho^{-(4+1)}\Re\left(Y_{4,2}\right)}
\end{alignat}
\end{subequations}
In this case, we get merely $\mathbb N(c,0)=0$ and hence $\mathbf X_i(c,0)=\mathbf 0$ ($i=1,\ldots,4$) in identity \eqref{identity:lie_brackets}.
Focusing now on the second term in the right hand side of \eqref{identity:lie_brackets}, it remains to compute, for all $i,j=1,\ldots,4$ and $\mathbf c=(c,0)$:
\begin{align}
  \partial_{s_i}\mathbf X_j(\mathbf c)-\partial_{s_j}\mathbf X_i(\mathbf c)&=
  \mathbb M(\mathbf c)^{-1}\Big[(\partial_{s_j}\mathbb M(\mathbf c)\mathbf X_i(\mathbf c)-\partial_{s_i}\mathbb M(\mathbf c)\mathbf X_j(\mathbf c))\nonumber \\
  &\hspace{5cm}+(\partial_{s_j}\mathbb N(\mathbf c)\mathbf f_i-\partial_{s_i}\mathbb N(\mathbf c)\mathbf f_j)\Big].\nonumber\\
  &=
  \mathbb M(\mathbf c)^{-1}\Big[\partial_{s_j}\mathbb N(\mathbf c)\mathbf f_i-\partial_{s_i}\mathbb N(\mathbf c)\mathbf f_j\Big].  \label{lie_brackets}
\end{align}
In particular, we need the expressions of the derivatives of the entries of the matrix $\mathbb N(\mathbf c)$ with respect to $s$. 
\begin{lemma}\label{lemma:DMN}
  Let $\mathbf V\in C^1_0(\mathbf R^3)^3\cap C^\infty(\mathbf R^3)^3$ and $\mathbf w_0\in C^\infty(\Sigma)^3$ (recall that $\Sigma$ is the boundary of the unit ball $B$ and $F:=\mathbf R^3\setminus\bar B$).
  For every $t$ small enough, we define $\varTheta_t={\rm Id}+t\mathbf V$, $\mathcal B_t=\varTheta_t(B)$, $\Sigma_t:=\partial\mathcal B_t$, $\mathcal F_t=\mathbf R^3\setminus\overline{\mathcal B_t}$ and $\mathbf w_t=\mathbf w_0\circ {\varTheta_t}^{-1}\in C^\infty(\Sigma_t)$.
  
 Let also $(\mathbf{u}_t,p_t)$ and $(\mathbf u^i_t,p^i_t)\in (W^1_0(\mathcal{F}_t))^3\times L^2(\mathcal F_t)$ ($i=1,\ldots,6$) be the solutions to the Stokes problems 
  in $\mathcal F_t$ with boundary conditions $\mathbf u_t=\mathbf w_t$, $\mathbf u^i_t(x)=x\times \mathbf e_i$ if $i\in\{1,2,3\}$ and $\mathbf u^i_t(x)= \mathbf e_{i-3}$ if $i\in\{4,5,6\}$ on $\Sigma_t$.
 Then we have 
 $$\frac{d}{dt}\left(\int_{\mathcal F_t}D(\mathbf u_t):D(\mathbf u^i_t)\, {\rm d}x\right)\Big|_{t=0}=\int_{F}D(\mathbf u'_0):D(\mathbf u^i_{t=0})\, \mathrm dx,$$ 
 where  $\mathbf u'_0\in (W^1_0(F))^3$ is solution of the homogeneous Stokes problem in $F$ with the boundary condition
  \begin{equation}\label{eq:bcDu}
    \mathbf u'_0 = -\nabla \mathbf u_{t=0}\mathbf V\qquad\text{on }\Sigma.
  \end{equation}
\end{lemma}
\begin{proof}
Since, for all $t$ small, the solution $\mathbf u_t$ is smooth, 
according to \cite[Theorem 4]{Simon:1991aa}, the derivative of $t\mapsto \mathbf u_t$ at $t=0$ is solution of the homogeneous Stokes problem  in $F$ with boundary condition \eqref{eq:bcDu} (notice that the boundary condition is merely obtained by differentiating the equation $\mathbf u_t\circ \varTheta_t=0$ with respect to $t$ at $t=0$).
Using the same argument as in the proof of Lemma \ref{lemma:FT}, we have 
  $2\int_{\mathcal F_t}D(\mathbf u_t):D(\mathbf u^i_t)\, \mathrm dx=-\int_{\partial B(0,R)}\mathbb T(\mathbf u_t,p_t)\tilde{\mathbf u}_i\cdot\mathbf n\, \mathrm d\sigma$. 
Differentiating with respect to $t$ and invoking the linearity of $\mathbb T$ and the Green formulae, we get the conclusion.
\end{proof}

Applying Lemma~\ref{lemma:DMN} with the vector fields $\mathbf V_i$ ($i=1,\ldots,4$) defined in \eqref{def:vectors}, we obtain after lengthy computations involving spherical harmonics:
$$\partial_{s_1}\mathbb{N}(c,s)\Big|_{s=0} =\begin{pmatrix}
  0 & 0 & 0 & 0\\
  0 & 0 & -\frac{3\sqrt{5}}{2^{\frac{7}{2}}} & 0\\
  0 & -\frac{3}{8} & 0 & 0\\
  0 & 0 & 0 & -\frac{\sqrt{3}\sqrt{5}}{\sqrt{2}\sqrt{7}}\\
  0 & 0 & 0 & 0\\
  0 & 0 & 0 & 0
\end{pmatrix},\quad \partial_{s_2} \mathbb{N}(c,s)\Big|_{s=0} =\begin{pmatrix}
  0 & 0 & -\frac{3\sqrt{5}}{2^{\frac{7}{2}}} & 0\\
  0 & 0 & 0 & 0\\
  \frac{3}{8} & 0 & 0 & 0\\
  0 & 0 & 0 & 0\\
  0 & 0 & 0 & \frac{\sqrt{3}\sqrt{5}}{\sqrt{2}\sqrt{7}}\\
  0 & 0 & 0 & 0
\end{pmatrix},$$
$$\partial_{s_3} \mathbb{N}(c,s)\Big|_{s=0} =\begin{pmatrix}
  0 & \frac{3\sqrt{5}}{2^{\frac{7}{2}}} & 0 & 0\\
  \frac{3\sqrt{5}}{2^{\frac{7}{2}}} & 0 & 0 & 0\\
  0 & 0 & 0 & 0\\
  0 & 0 & 0 & 0\\
  0 & 0 & 0 & 0\\
  0 & 0 & 0 & -\frac{2\sqrt{3}}{\sqrt{7}}
\end{pmatrix},\quad\partial_{s_4}\mathbb{N}(c,s)\Big|_{s=0} =\begin{pmatrix}
  0 & 0 & 0 & 0\\
  0 & 0 & 0 & 0\\
  0 & 0 & 0 & 0\\
  -\frac{\sqrt{3}\,5^{\frac{3}{2}}}{2^{\frac{9}{2}}\sqrt{7}} & 0 & 0 & 0\\
  0 & \frac{\sqrt{3}\,5^{\frac{3}{2}}}{2^{\frac{9}{2}}\sqrt{7}} & 0 & 0\\
  0 & 0 & -\frac{5\sqrt{3}}{8\sqrt{7}} & 0
\end{pmatrix}.$$

One easily check now that ${\rm dim}\, \left({\rm span}\left\{\partial_{s_i}\mathbb N(\mathbf c)\mathbf f_j- \partial_{s_j}\mathbb N(\mathbf c)\mathbf f_i,1\leq i<j<4\right\}\right)=6$ and then 
 ${\rm dim}\, \left({\rm span}\left\{\mathbf Z^k_c({\rm Id},0), [\mathbf Z^i_c({\rm Id},0), \mathbf Z^j_c({\rm Id},0)],1\leq k\leq 4,1\leq i<j<4\right\}\right)=10$ which is the dimension of ${\rm SO}(3)\times\mathbf R^3\times\mathcal S(c)$.
 It entails, according to the forth point of Proposition~\ref{properties_of_SS}:

\begin{proposition}
\label{everything_is_controllable}
For any integer $n\geq 4$, the set of all the controllable SS is an open dense subset in $\mathcal C(n)$. 
\end{proposition}

\section{Proofs of the Main Results}
\label{sec:main_results}
\subsubsection*{Proof of Proposition~\ref{existence}}
Let a control function $\vartheta$ be given in $W^{1,1}([0,T],D^1_0(\mathbf R^3))$ and denote $\varTheta:={\rm Id}+\vartheta$. With the notation of Lemma~\ref{lemma:import}, at any time $t$ the entries of the matrix $\mathbb M(t)$ have the form $\varPhi(\mathbf q)$ with $\mathbf q:=(\vartheta,\mathcal W)$, $\mathcal W:=(\mathbf W^1,\mathbf W^2)$, $\mathbf W^j\in\{\mathbf e_i\times\varTheta_t,\,\mathbf e_i,\,i=1,2,3\}$ ($j=1,2$). We deduce that $t\in[0,T]\mapsto\mathbb M(t)\in{\rm M}(3)$ is absolutely continuous. To get the expression of the elements of the vector $\mathbf N(t)$ we only have to modify $\mathbf W^2$ which has to be equal to $\partial_t\vartheta_t$. It entails that $t\in[0,T]\mapsto\mathbf N(t)\in\mathbf R^6$ is in $L^1([0,T],\mathbf R^6)$. 
Existence  of solutions is now straightforward because $t\in[0,T]\mapsto\mathbb M(t)^{-1}\mathbf N(t)\in\mathbf R^6$ is in $L^1([0,T],\mathbf R^6)$ and Carath\'eodory's existence theorem applies to \eqref{complement}. Uniqueness derives from Gr\"onwall's inequality.

Let us address the stability result. With the same notation as in the statement of Proposition~\ref{existence}, denote by $({\boldsymbol\Omega}^j,\mathbf v^j)^\ast$ the left hand side of identity \eqref{dynamics} when the control is $\vartheta^j$ and $(\bar{\boldsymbol\Omega},\bar{\mathbf v})^\ast$ when the control is $\bar\vartheta$.  As $j\to+\infty$, it is clear that $({\boldsymbol\Omega}^j,\mathbf v^j)^\ast\to(\bar{\boldsymbol\Omega},\bar{\mathbf v})^\ast$ in $L^1([0,T],\mathbf R^6)$. Then, integrating \eqref{complement} between 0 and $t$ for any $0\leq t\leq T$, we get the estimate $\|\bar R(t)-R^j(t)\|_{{\rm M}(3)}\leq \int_0^T\|\bar R(s)-R^j(s)\|_{{\rm M}(3)}\|\bar{\boldsymbol\Omega}(s)\|_{\mathbf R^3}+\|\boldsymbol\Omega^j(s)-\bar{\boldsymbol\Omega}(s)\|_{\mathbf R^3}{\rm d}s.$ Applying Gr\"onwall's inequality, we conclude that $R^j\to\bar R$ in $C([0,T],{\rm M}(3))$ as $j\to+\infty$ and we use again the ODE to prove that $\dot R^j\to\dot{\bar R}$ in $L ^1([0,T],{\rm M}(3))$. Then, it is easy to obtain the convergence of $\mathbf r^j$ to $\bar{\mathbf r}$ and to conclude the proof.

\subsubsection*{Proof of Theorems~\ref{main_theorem_cont} and \ref{main_theorem_free}} We shall focus on the proof of Theorem~\ref{main_theorem_cont} because it will contain the proof of Theorem~\ref{main_theorem_free}. For any integer $n$, we shall use the notation $\|c\|_{\mathcal C(n)}:=\|\vartheta\|_{C^1_0(\mathbf R^3)^3}+\sum_{i=1}^n\|\mathbf V_i\|_{C^1_0(\mathbf R^3)^3}$
for all $c\in C^1_0(\mathbf R^3)^3\times(C^1_0(\mathbf R^3)^3)^n$ with, as usual, $c:=(\vartheta,\mathcal V)$ and $\mathcal V:=(\mathbf V_1,\ldots,\mathbf V_n)$.

Let $\varepsilon>0$ and the functions $t\in[0,T]\mapsto\bar\vartheta_t\in D^1_0(\mathbf R^3)$ and $t\in[0,T]\mapsto (\bar R(t),\bar{\mathbf r}(t))\in{\rm SO}(3)\times\mathbf R^3$ be given as in the statement of the theorem. According to Proposition~\ref{PROPo}, we can assume that $\bar\vartheta\in C^\omega([0,T],D^1_0(\mathbf R^3))\cap\mathcal A$ because this space is a dense subspace of $\mathcal A$. 

\noindent{\bf Step 1 (small initial jerking of the swimmer).} In this step, we prove that the swimmer is able to modify slightly its shape in order to become controllable.
Set $\bar\vartheta^1:=\bar\vartheta_{t=0}$ and $\bar{\mathbf V}^1_1:=\partial_t\bar\vartheta_{t=0}\in C^1_0(\mathbf R^3)^3$. According to the self-propelled constraints \eqref{self-propelled-cond}, it is always possible to find three elements $\bar{\mathbf V}^1_j$ ($j=2,3,4$) in $C^1_0(\mathbf R^3)^3$ such that the SS $\bar c^1:=(\bar\vartheta^1,\bar{\mathcal V}^1)$ belongs to ${\mathcal C}(4)$ (with $\bar{\mathcal V}^1:=(\bar{\mathbf V}^1_1,\ldots,\bar{\mathbf V}^1_4)$). Then, Proposition~\ref{everything_is_controllable} guarantees that for any $\delta>0$ it is possible to find a {\it controllable} SS  in $\mathcal C(4)$, denoted by $c^1:=(\vartheta^1,\mathcal V^1)$ where $\mathcal V^1:=(\mathbf V^1_1,\ldots,\mathbf V^1_4)$, such that $\|c^1-\bar c^1\|_{\mathcal C(4)}<\delta/2$ ($\delta>0$ is meant to be small an will be fixed later on). Moreover, we claim that $c^1$ can be chosen in such a way that there exists a smooth {\it allowable} function (i.e. satisfying \eqref{self-propelled-cond}) $t\in[-1,0]\mapsto\vartheta^0_t\in D^1_0(\mathbf R^3)$ such that $\vartheta^0_{t=-1}=\bar\vartheta^1$ and $\vartheta^0_{t=0}=\vartheta^1$ (i.e. the swimmer can modify its shape from $\bar\vartheta^1$ into $\vartheta^1$ by self-deforming on time interval $[-1,0]$). Indeed, denote $\hat c^1:=(\hat\vartheta^1,\hat{\mathcal V}^1)\in \mathcal C(4)$ a controllable SS such that $\|\bar c^1-\hat c^1\|_{\mathcal C(4)}$ be small. Then define $\bar\vartheta^0_t:=\bar\vartheta^1+(1+t)(\hat\vartheta^1-\bar\vartheta^1)$ for every $t\in[-1,0]$. Since $D^1_0(\mathbf R^3)$ is open, for $\|\hat\vartheta^1-\bar\vartheta^1\|_{C^1_0(\mathbf R^3)^3}$ small enough, $\bar\vartheta^0_t$ will remain in $D^1_0(\mathbf R^3)$ for all $t\in[-1,0]$. Then, Proposition~\ref{PROP:1_1} asserts that there exists a function $Q_0\in AC([-1,0],{\rm SO}(3))$ and an allowable shape function $\vartheta^0_t\in W^{1,1}([-1,0],D^1_0(\mathbf R^3))$ such that $\varTheta^0_t$ links $\bar\varTheta^1$ (at $t=-1$) to some $\varTheta^1$  (at $t=0$) satisfying $\varTheta^1|_\Sigma=Q_0(0)\hat\varTheta^1|_\Sigma$. A careful reading of the proof of Proposition~\ref{PROP:1_1} allows noticing that  $\| Q_0-{\rm Id}\|_{C([-1,0],{\rm M}(3))}$ and $\|\vartheta^0_t-\bar\vartheta^1\|_{W^{1,1}([-1,0],D^1_0(\mathbf R^3))}$ go to 0 as $\|\hat\vartheta^1-\bar\vartheta^1\|_{C^1_0(\mathbf R^3)^3}$ goes to 0. Set now $\vartheta^1:=\varTheta^1-{\rm Id}$, $\mathbf V^1_i:=Q_0(0)\hat{\mathbf V}^1_i$ ($i=1,\ldots,4$) and observe that the resulting SS $c^1$ satisfies the requirements. Furthermore, according to Proposition~\ref{existence}, $\|\hat\vartheta^1-\bar\vartheta^1\|_{C^1_0(\mathbf R^3)^3}$ can always be made small enough for the control function $\vartheta^0_t$ to produce a rigid displacement $t\in[-1,0]\mapsto (R_0(t),\mathbf r_0(t))\in{\rm SO}(3)\times\mathbf R^3$ satisfying $\sup_{t\in[-1,0]}\big(\|R_0(t)-\bar R(0)\|_{{\rm M}(3)}+\|\mathbf r_0(t)-\bar{\mathbf r}(0)\|_{\mathbf R^3}+\|\vartheta^0_t-\bar\vartheta^1\|_{C^1_0(\mathbf R^3)^3}\big)<\varepsilon/2$. Eventually, remark that this step of initial jerking performed on the time interval $[t_0,t_1]:=[-1,0]$ can actually be carried out on a time interval arbitrarily short just by rescaling the time.

\noindent{\bf Step 2 (building a continuous piecewise $C^1$ control function).}
Since the function $\partial_t\bar\vartheta$ is continuous on the compact set $[0,T]$, it is uniformly continuous. For any $ \nu>0$, there exists $\delta_ \nu>0$ such that 
$ \|\partial_t\bar\vartheta_t-\partial_t\bar\vartheta_{t'}\|_{C^1_0(\mathbf R^3)^3}<  \nu$ providing that $|t-t'|\leq \delta_ \nu$.
Then, we divide the time interval $[0,T]$ into $0=t_1<t_2<\ldots<t_k=T$ such that $|t_{j+1}-t_j|<\delta_ \nu$ for $j=1,\ldots,k-1$.
For any $t\in[t_1,t_2]$, we have the estimate:
\begin{multline*}
\|\bar\vartheta_t-(\vartheta^1+(t-t_1)\mathbf V_1^1)\|_{C^1_0(\mathbf R^3)^3}\leq\|\bar\vartheta_t-(\bar\vartheta^1+(t-t_1)\bar{\mathbf V}^1_1)\|_{C^1_0(\mathbf R^3)^3}\\
+\|\bar\vartheta^1-\vartheta^1\|_{C^1_0(\mathbf R^3)^3}+(t-t_1)\|\bar{\mathbf V}^1_1-\mathbf V_1^1\|_{C^1_0(\mathbf R^3)^3}.
\end{multline*}
On the one hand, we have, for all $t\in [t_1,t_2]$, $\|\bar\vartheta_t-(\bar\vartheta^1+(t-t_1)\bar{\mathbf V}^1_1)\|_{C^1_0(\mathbf R^3)^3}<  \nu |t-t_1|$. 
On the other hand, still for $t_1\leq t\leq t_2$ and if we assume that $\delta_ \nu<1$, we get $\|\bar\vartheta^1-\vartheta^1\|_{C^1_0(\mathbf R^3)^3}+(t-t_1)\|\bar{\mathbf V}^1_1-\mathbf V_1^1\|_{C^1_0(\mathbf R^3)^3}\leq \delta/2.$
We denote $\bar\vartheta^2:=\bar\vartheta_{t=t_2}$. It is always possible to supplement $\bar{\mathbf V}^2_1:=\partial_t\bar\vartheta_{t_2}$ with vector fields $\bar{\mathbf V}_j^2$ ($j=2,\ldots,4$) in such a way that $\bar c^2:=(\bar\vartheta^2,\bar{\mathcal V}^2)$ be in ${\mathcal C}(4)$ with the obvious notation $\bar{\mathcal V}^2:=(\bar{\mathbf V}^2_1,\ldots,\bar{\mathbf V}_4^2)$. 
We define $\vartheta^2:=\vartheta^1+(t_2-t_1)\mathbf V^1_1$. For any $t_1\leq t\leq t_2$, Proposition~\ref{properties_of_SS} guarantees that the SS $c^1_t:=(\vartheta^1+(t-t_1)\mathbf V^1_1,\mathcal V^1)$ is controllable. In particular, for $t=t_2$, there exists an integer $k$ and a family of 10 vector fields\footnote{10 is the dimension of ${\rm SO}(3)\times\mathbf R^3\times\mathcal S(c^1_{t_2})$} in $\mathcal E_k$ (the set of all the Lie brackets of order lower or equal to $k$) such that the determinant of the family is nonzero. But this determinant can be thought of as an analytic function in $\mathcal V^1$. The set $\varPi^{-1}(\{\vartheta^2\})$ being an analytic connected submanifold of $(C^1_0(\mathbf R^3)^3)^4$ (see Corollary~\ref{cor:2}), the determinant is nonzero everywhere on this set but maybe in a closed subset of empty interior (for the induced topology). Therefore, it is possible to find $\mathcal V^2\in (C^1_0(\mathbf R^3)^3)^4$ such that the SS $c^2:=(\vartheta^2,\mathcal V^2)$ is controllable and $\|\bar c^2-c^2\|_{\mathcal C(4)}<(\delta/2+ \nu (t_2-t_1))+\delta/4$. 

 By induction, we can build $\bar c^j$ and $c^j$ ($j=1,2,\ldots,k$) such that  (i) $\|\bar c^j-c^j\|\leq \delta/2+\sum_{i=2}^k{\delta}/{2^i}+ \nu (t_{i}-t_{i-1})<\delta+ \nu T$ and (ii) every $c^j$ is controllable. We choose $\delta$ and $ \nu$ in such a way that $\delta+ \nu T<\varepsilon/4$ and we define $t:[0,T]\mapsto \tilde\vartheta_t\in D^1_0(\mathbf R^3)$  as continuous, piecewise affine functions by $\tilde\vartheta_t:=\vartheta^j+(t-t_j)\mathbf V^j_1$ if $t\in [t_j,t_{j+1}]$ ($j=1,\ldots,k-1$).  Notice that for any $t\in[0,T]$, $\|\bar\vartheta_t-\tilde\vartheta_t\|_{C^1_0(\mathbf R^3)^3}<\varepsilon/2$.
 
Definition~\ref{def:cont:SS} 	and Proposition~\ref{prop:prop12} ensure that, on every   interval $[t_j,t_{j+1}]$ ($j=1,\ldots,k-1$), there exist four $C^1$ functions $\lambda^j_i:[t_j,t_{j+1}]\mapsto \mathbf R$ ($i=1,\ldots,4$) such that the solution $(R_j,\mathbf r_j,s^j):[t_j,t_{j+1}]\to {\rm SO}(3)\times\mathbf R^3\times\mathbf R^4$ to the ODE \eqref{dynamics:3}  with vector fields $\mathbf Z_{c^j}^i(R_j,s^j)$ and Cauchy data $R_1(t_1) = R_0(0)$, $\mathbf r_1(t_1)=\mathbf r_0(0)$, $R_j(t_j)=\bar R(t_j)$, $\mathbf r_j(t_j)=\bar{\mathbf r}(t_j)$ ($j=2,\ldots,k-1$) and $s^j(t_j)=0$ ($j=1,\ldots,k-1$) satisfies:
\begin{enumerate}
\item 
$\sup_{t\in[t_j,t_{j+1}]}\Big(\| \bar R(t)-R_j(t)\|_{{\rm M}(3)}+\|\bar{\mathbf r}(t)-\mathbf r_j(t)\|_{\mathbf R^3}+\|\tilde\vartheta_t-\vartheta_t^j\|_{C^1_0(\mathbf R^3)^3}\Big)<\varepsilon/4$ with $\vartheta^j_t:=\vartheta^j+\sum_{i=1}^4s^j_i(t)\mathbf V^j_i$;
\item $R_j(t_{j+1})=\bar R(t_{j+1})$, $\mathbf r_j(t_{j+1})=\bar{\mathbf r}(t_{j+1})$ and $s^j(t_{j+1})=(t_{j+1}-t_j,0,0,0)^\ast
$.
\end{enumerate}
 With these settings, the functions $t\in[-1,T]\mapsto\breve\vartheta_t\in D^1_0(\mathbf R^3)$, $\breve R:[-1,T]\to{\rm SO}(3)$ and $\breve{\mathbf r}:[-1,T]\to\mathbf R^3$ defined by $\breve\vartheta_t:=\vartheta^j_t$, $\breve R(t):=R_j(t)$ and $\breve{\mathbf r}(t):=\mathbf r_j(t)$ if $t\in[t_j,t_{j+1}]$ ($j=0,\ldots,k-1$) are continuous, piecewise $C^1$. 
 
 \noindent{\bf Step 3 (smoothing the control function).}
 We obtain a control function on $[0,T]$ (still denoted by $\breve\vartheta$) by merely shifting/rescaling the time, from $[-1,T]$ onto $[0,T]$. Beforehand and as already mentioned, the first time interval $[t_0,t_1]:=[-1,0]$ could have been shortened as much as necessary for the estimate 
$$\sup_{t\in[0,T]}\Big(\| \bar R(t)-\breve R(t)\|_{{\rm M}(3)}+\|\bar{\mathbf r}(t)-\breve{\mathbf r}(t)\|_{\mathbf R^3}+\|\bar\vartheta_t-\breve\vartheta_t\|_{C^1_0(\mathbf R^3)^3}\Big)<\varepsilon/2,$$
to be true after the shifting/rescaling process. Then, we invoke Proposition~\ref{PROPo} and Proposition~\ref{existence} to conclude that there exists $t\in[0,T]\mapsto\vartheta_t\in D^1_0(\mathbf R^3)$ analytic, satisfying \eqref{self-propelled-cond} and such that 
$$\sup_{t\in[0, T]}\Big(\| R(t)-\breve R(t)\|_{{\rm M}(3)}+\|{\mathbf r}(t)-\breve{\mathbf r}(t)\|_{\mathbf R^3}+\|\vartheta_t-\breve\vartheta_t\|_{C^1_0(\mathbf R^3)^3}\Big)<\varepsilon/2,$$
where $(R,\mathbf r):[0,T]\mapsto {\rm SO}(3)\times\mathbf R^3$ is the solution to System~\eqref{main_dynamics} with initial data $(R(0),\mathbf r(0))=(\bar R(0),\bar{\mathbf r}(0))$ and control $\vartheta$. The proof is then complete. 
\section{Conclusion}
\label{sec:conclusion}
In this paper, we have proved that every 3D microswimmer as the ability to swim (i.e. not only moving but  tracking any given trajectory). Moreover, this can be achieved by means of arbitrarily small shape changes which can be superimposed to any preassigned macro deformation.
When the shape changes are expressed as a finite combination of elementary deformations (and no macro shape changes are prescribed), we have shown that only four elementary deformations are needed for the swimmer to be able to track any trajectory. In this case and when the rate of shape changes (i.e. the velocity of deformations) is valued in a compact set, an optimal control exists for a wide variety of cost functionals.
\appendix

\section{Function spaces}
\label{SEC:diffeo}
\subsubsection*{Classical function spaces} 
\begin{itemize}
\item  For any open set $\Omega\subset\mathbf R^3$ (included $\Omega=\mathbf R^3$), $\mathcal D(\Omega)$ is the space of the smooth ($C^\infty$) functions, compactly supported in $\Omega$. 
\item For any open set $\Omega\subset\mathbf R^3$ (included $\Omega=\mathbf R^3$), the set $C^1_0(\Omega)$ 
is the completion of $\mathcal D(\Omega)$ for the norm $\|u\|_{C^1_0(\Omega)}:=\sup_{x\in\Omega}|u(x)|+\|\nabla u(x)\|_{\mathbf R^3}$. When $\Omega=\mathbf R^3$, we get $C^1_0(\mathbf R^3):=\{u\in C^1(\mathbf R^3)\,:\:|u(x)|\to 0\text{ and }\|\nabla u(x)\|_{\mathbf R^3}\to 0\text{ as }\|x\|_{\mathbf R^3}\to+\infty\}$.
\item The space $C^1_0(\mathbf R^3)^3$ is the Banach space of all of the vector fields in $\mathbf R^3$ whose every component belongs to $C^1_0(\mathbf R^3)$.
\item For any Banach space $E$ and any $T>0$, $C^\omega([0,T],E)$ is the space of analytic functions on $[0,T]$, valued in $E$.
\item Let now $E$ be an open subset or an embedded submanifold of an Euclidean space and $T>0$, then $AC([0,T],E)$ consists in the absolutely continuous functions from $[0,T]$ into $E$. It is endowed with the norm 
$\|u\|_{AC([0,T],E)}:=\sup_{t\in[0,T]}\|u_t\|_E+\int_0^T\|\partial_t u_t\|_E{\rm d}t$.
\item $C_0^m(\Omega,{\rm M}(k))$ ($m$ an integer) is the Banach space of the  functions of class $C^m$ in $\mathbf R^3$ valued in ${\rm M}(k)$ (${\rm M}(k)$ stands for the Banach space of the $k\times k$ matrices, $k$ a positive integer) and compactly supported in $\Omega$. 
\item $E^m_0(\Omega,{\rm M}(k))$ stands for the connected component containing the zero function of the open subset $\{M\in C^m_0(\Omega,{\rm M}(k))\,:\det({\rm Id}+M(x))\neq 0\,\,\forall\,x\in\mathbf R^3\}$. 
\end{itemize}
\begin{lemma}
The set $\tilde D^1_0(\mathbf R^3):=\{\vartheta\in C^1_0(\mathbf R^3)^3\;{\rm s.t.}\;{\rm Id}+\vartheta{\rm\; is\; a }\;C^1\;{\rm diffeomorphism\; of\; }\mathbf R^3\}$
is open in $C^1_0(\mathbf R^3)^3$.
\end{lemma}
\begin{proof}
The mapping $\vartheta\in C^1_0(\mathbf R^3)^3\mapsto\delta_{\vartheta}:=\inf_{\mathbf e\in S^2\atop x\in\mathbf R^3}\langle{\rm Id}+ \nabla \vartheta(x),\mathbf e\rangle\cdot \mathbf e\in\mathbf R$ ($S^2$ stands for the unit 2 dimensional sphere)
is well defined and continuous. For any $\vartheta_0\in \tilde D^1_0(\mathbf R^3)$, we have $\delta_{\vartheta_0}>0$ and for all $x,y\in\mathbf R^3$ and $\mathbf e:=(y-x)/|y-x|$ the following estimate holds:
$(y+\vartheta(y)-x-\vartheta(x))\cdot \mathbf e=|y-x|\int_0^1\langle {\rm Id}+\nabla\vartheta(x+t\mathbf e),\mathbf e\rangle\cdot \mathbf e\,{\rm d}t>|y-x|\delta_\vartheta$.
We deduce that ${\rm Id}+\vartheta$ is one-to-one if $\vartheta$ is close enough to $\vartheta_0$. Further, still for $\vartheta$ close enough to $\vartheta_0$, ${\rm Id}+\vartheta$ is a local diffeomorphism (according to the local inversion Theorem) and hence it is onto.
\end{proof}

\begin{definition}
\label{def_D10}
We denote $D^1_0(\mathbf R^3)$ the connected component of $\tilde D^1_0(\mathbf R^3)$ that contains the identically zero function.
\end{definition}

 If $\vartheta\in C^1_0(\mathbf R^3)^3$ is such that $\|\vartheta\|_{C^1_0(\mathbf R^3)^3}<1$, the local inversion Theorem and a fixed point argument ensure that $\rm{Id}+\vartheta$ is a $C^1$ diffeomorphism so we deduce that $D^1_0(\mathbf R^3)$ contains the unit ball of $C^1_0(\mathbf R^3)^3$.

\subsubsection*{Sobolev spaces} 
\begin{itemize}
\item We define the weight function $\theta(x):=\sqrt{1+|x|^2}$ ($x\in\mathbf R^3$) and the weighted Sobolev spaces:
\begin{align}
W^1_0(\mathcal F)&:=\big\{u\in \mathcal D'(\mathcal F)\,:\, \theta^{-1}u\in L^2(\mathcal F)\big\},\\
\stackrel{\circ}{W_0^1}(\mathcal F)&:=\big\{u\in W^1_0(\mathcal F)\,:\,\gamma_\Sigma(u)=0\big\},
\end{align}
where $\gamma_\Sigma:W^1_0(\mathcal F)\mapsto  H^{1/2}(\Sigma)$ is the classical trace operator. The dual space of $\stackrel{\circ}{W_0^1}(\mathcal F)$ is $W^{-1}_0(\mathcal F)$. 
\item For any Banach space $E$, $W^{1,1}([0,T],E)$ is the Bochner-Sobolev spaces (see for instance \cite[\S 7.1, page 187]{Roubicek:2005aa}) consisting in the functions $u:[0,T]\mapsto E$ measurable and such that $u$ and $u'$ belong to $L^1([0,T],E)$ (the derivative $u'$ as to be understood in the sense of the distributions).  It can be proved that $W^{1,1}([0,T],E)$ is continuously embedded in $C([0,T],E)$ and that $u(t)=u(0)+\int_0^t u'(s)\,{\rm d}s$ for all $t\in[0,T]$ and all $u\in W^{1,1}([0,T],E)$, where the integral is a Bochner integral (a generalization to Banach space valued functions of the Lebesgue integral). The space $W^{1,1}([0,T],E)$ is endowed with the norm $\|u\|_{W^{1,1}([0,T],E)}:=\|u\|_{C([0,T],E)}+\int_0^T\|u'(s)\|_E\,{\rm d}s$.
\end{itemize}

\section{Control functions smoothing}
\begin{proposition}
\label{PROPo}
For every $\varepsilon>0$ and every $\vartheta\in\mathcal A$, there exists $\bar\vartheta\in C^\omega([0,T],D^1_0(\mathbf R^3))\cap\mathcal A$   such that $\|\bar\vartheta-\vartheta\|_{W^{1,1}([0,T],D^1_0(\mathbf R^3))}<\varepsilon$ and $\bar\vartheta_{t=0}=\vartheta_{t=0}$. In particular $C^\omega([0,T],D^1_0(\mathbf R^3))\cap\mathcal A$ is dense in $\mathcal A$.
\end{proposition}

\begin{proof}
Let $\vartheta$ be in $\mathcal A$. Since, $C^\omega([0,T],C^1_0(\mathbf R^3)^3)$ is dense in $L^1([0,T],C^1_0(\mathbf R^3)^3)$, we can always pick an element $\zeta\in C^\omega([0,T],C^1_0(\mathbf R^3)^3)$ which makes $\|\zeta-\partial_t\vartheta\|_{L^1([0,T],C^1_0(\mathbf R^3)^3)}$ as small as required. Define for every $t\in[0,T]$ the analytic function $\tilde\vartheta_t=\vartheta_{t=0}+\int_0^t\zeta(s)\,{\rm d}s$, keeping in mind that the quantity $\|\tilde\vartheta-\vartheta\|_{W^{1,1}([0,T],D^1_0(\mathbf R^3))}$ can be made arbitrarily small.
Following the lines of the proof of Proposition~\ref{PROP:1_1}, we define $\tilde{\mathbf s}(t):=(1/4\pi)\int_\Sigma \tilde\varTheta_t\,{\rm d}\sigma$, $\tilde\varTheta^\dag_t:=\tilde\varTheta_t-\tilde{\mathbf s}(t)$, the matrix $\tilde{\mathbb J}(t):=\int_\Sigma\|\tilde\varTheta^\dag_t\|^2_{\mathbf R^3}{\rm Id}-\tilde\varTheta^\dag_t\otimes\tilde\varTheta^\dag_t{\rm d}\sigma$ (positive definite for every $t\in[0,T]$) and $\tilde{\boldsymbol\chi}(t):=\tilde{\mathbb J}(t)^{-1}\int_\Sigma\partial_t\tilde\varTheta^\dag_t\times\tilde\varTheta^\dag_t\,{\rm d}\sigma$. Let us introduce as well ${\mathbb J}(t):=\int_\Sigma\|\varTheta_t\|^2_{\mathbf R^3}{\rm Id}-\varTheta_t\otimes\varTheta_t{\rm d}\sigma$. Observing again that the quantity $\|\tilde\vartheta-\vartheta\|_{W^{1,1}([0,T],D^1_0(\mathbf R^3))}$ can be arbitrarily small, we draw the same conclusion for 
$\|\tilde{\mathbf s}\|_{W^{1,1}([0,T],\mathbf R^3)}$, then for $\|\tilde{\mathbb J}(t)^{-1}-\mathbb J(t)^{-1}\|_{C^0([0,T],M(3))}$ and finally for $\|\tilde{\boldsymbol\chi}\|_{L^1([0,T],\mathbf R^3)}$. We infer that $\|Q-{\rm Id}\|_{W^{1,1}([0,T],M(3))}$, the solution to the Cauchy problem $\partial_tQ_t=Q_t \tilde{\boldsymbol\chi}(t)$ with initial data $Q_{t=0}={\rm Id}$ is arbitrarily small as well. Then we set 
$\varTheta^\star_t=Q(t)\tilde\varTheta_t^\dag$. At this point, $\varTheta^\star$ satisfies \eqref{self-propelled-cond} but $\vartheta^\star_t$ (for $t\in[0,T]$) is unlikely in $D^1_0(\mathbf R^3)$ (because $\varTheta^\star(x)\to Q(t)(x-\tilde{\mathbf s}(t))\neq x$ as $\|x\|_{\mathbf R^3}\to+\infty$). Notice however that for every smooth compactly supported function $\xi:\mathbf R^3\to\mathbf R$, the quantity $\|\xi(\vartheta^\star-\vartheta)\|_{W^{1,1}([0,T],D^1_0(\mathbf R^3))}$ can be made small. Let $\Omega$ and $\Omega'$ be large balls such that $\bigcup_{t\in[0,T]}\varTheta^\star_t(\bar B)\subset\Omega$ and $\bar\Omega\subset\Omega'$ and define $\xi$ as a cut-off function valued in $[0,1]$ and such that  $\xi =1$ in $\Omega$ and $\xi=0$ in $\mathbf R^3\setminus\bar\Omega'$. To complete the proof, define $\bar\varTheta$ as the flow associated with the Cauchy problem $\dot X(t,x)=\xi(x)\partial_t\vartheta^\star_t(x)+(1-\xi(x))\partial_t\vartheta_t(x)$, $X(0,x)=\varTheta_{t=0}(x)$. Indeed, $\|\vartheta-\bar\vartheta\|_{W^{1,1}([0,T],D^1_0(\mathbf R^3))}$ goes to 0 as $\|\xi(\partial_t\vartheta^\star_t-\partial_t\vartheta)\|_{L^1([0,T],C^1_0(\mathbf R^3)^3)}$ goes to $0$.
\end{proof}
\section{Stokes Problem and Change of Variables}
\label{SEC:Stokes}
\subsection{Well-posedness of the Stokes problem in an exterior domain}
The following results that can be found in \cite{Girault:1991aa}:
\begin{theorem}
\label{exist_stokes}
Let $\Sigma$ be connected an Lipschitz continuous. Then, 
for any $(\mathbf f,g,\mathbf h)\in (W_0^{-1}(\mathcal F))^3\times  L^2(\mathcal F)\times (H^{1/2}(\Sigma))^3$, there exists a unique pair $(\mathbf u,p)\in (W^1_0(\mathcal F))^3\times L^2(\mathcal F)$ such that:
\begin{subequations}
\begin{alignat}{3}
-\Delta \mathbf u+\nabla p&=\mathbf f&\quad&\text{in }\mathcal F,\\
\nabla\cdot \mathbf u&=g&& \text{in }\mathcal F,\\
\mathbf u&=\mathbf h&&\text{on }\Sigma.
\end{alignat}
\end{subequations}
The solution has to be understood in the weak sense, namely:
\begin{subequations}
\label{stokes_varia}
\begin{align}
\int_{\mathcal F}\!\!\nabla \mathbf u:\nabla \mathbf v\,{\rm d}x-\int_{\mathcal F}\!\!p\nabla\cdot \mathbf v)\,{\rm d}x&=\langle \mathbf f,\mathbf v\rangle_{(W^{-1}_0)^3\times (\stackrel{\circ}{W^1_0})^3},\quad\forall\,\mathbf v\in  
(\stackrel{\circ}{W^1_0}(\mathcal F))^3,\\
\nabla\cdot \mathbf u&=g\quad\text{in }\mathcal F,\\
\gamma_\Sigma(\mathbf u)&=\mathbf h\quad\text{on }\Sigma.  
\end{align}
\end{subequations}
Besides, there exists a constant $C_{\mathcal F}>0$ (depending on $\mathcal F$ only) such that:
$$\|\mathbf u\|_{ (W^1_0(\mathcal F))^3}+\|p\|_{L^2(\mathcal F)}\leq C_{\mathcal F}[\|\mathbf f\|_{(W^{-1}_0(\mathcal F))^3}+\|g\|_{L^2(\mathcal F)}+\|\mathbf h\|_{(H^{1/2}(\Sigma))^3}].$$
\end{theorem}
\subsection{Change of variables}
\label{SUB:change_of_variables}
We denote, for all  $\xi\in D^1_0(\bar B,\mathbf R^3)$, $\varUpsilon:={\rm Id}+\xi$,  $J_\xi:=\det(\nabla\varUpsilon)$ and we define the matrices $\mathbb A_\xi:=(\nabla\varUpsilon^\ast\nabla\varUpsilon)^{-1}J_\xi$ and $\mathbb B_\xi:=(\nabla\varUpsilon^\ast)^{-1} J_\xi$. 
\begin{proposition}
\label{prop:change_of_variables}
If $\Sigma$ is Lipschitz continuous, for all $\xi\in D^1_0(\bar B,\mathbf R^3)$ and for all $(\mathbf f,g,\mathbf h)\in (W^{-1}_0(\mathcal F))^3\times L^2(\mathcal F)\times (H^{1/2}(\mathcal F))^3$ the following problem:
\begin{subequations}
\label{stokes_change}
\begin{align}
\int_{\mathcal F}\!\!\nabla \mathbf U_\xi\mathbb A_\xi: \nabla \mathbf V\,{\rm d}x-
\int_{\mathcal F}\!\!P_\xi\mathbb B_\xi:\nabla \mathbf V\,{\rm d}x&=
\langle \mathbf f,\mathbf V\rangle_{(W^{-1}_0)^3\times (\stackrel{\circ}{W^1_0})^3},&\quad\forall\,\mathbf V\in (\stackrel{\circ}{W^1_0}(\mathcal F))^3,\\
\mathbb B_\xi:\nabla \mathbf U_\xi&=g\quad\text{in }\mathcal F,\\
\gamma_{\Sigma}(\mathbf U_\xi)&=\mathbf h\quad\text{on }\Sigma,
\end{align}
\end{subequations}
admits a unique solution $(\mathbf U_\xi,P_\xi)\in(W^1_0(\mathcal F))^3\times L^2(\mathcal F)$. Moreover, there exists a constant $C_\xi(\mathcal F)>0$ depending on $\mathcal F$ and $\xi$ only such that:
$$\|\mathbf U_\xi\|_{(W^1_0(\mathcal F))^3}+\|P_\xi\|_{L^2(\mathcal F)}\leq C_\xi(\mathcal F)[\|\mathbf f\|_{(W^{-1}_0(\mathcal F))^3}+\|g\|_{L^2(\mathcal F)}+\|\mathbf h\|_{(H^{1/2}(\Sigma))^3}].$$
\end{proposition}
\begin{proof}
Let us introduce $\mathcal F_\xi:=\varUpsilon(\mathcal F)$, $\Sigma_\xi:=\varUpsilon(\Sigma)$, $g_\xi:=g\circ\varUpsilon^{-1}/(J_\xi\circ\varUpsilon^{-1})$ and $\mathbf h_\xi:=\mathbf h\circ\varUpsilon^{-1}$. We denote by $\mathbf f_\xi$ the distribution in $(W^{-1}_0(\mathcal F_\xi))^{-1}$ defined by 
$$\langle \mathbf f_\xi,\varphi\rangle_{(W^{-1}_0(\mathcal F_\xi))^3\times (\stackrel{\circ}{W^1_0}(\mathcal F_\xi))^3}:=\langle \mathbf f,\varphi\circ\varUpsilon\rangle_{(W^{-1}_0(\mathcal F))^3\times (\stackrel{\circ}{W^1_0}(\mathcal F))^3},\quad\forall\,\varphi\in \stackrel{\circ}{W^1_0}(\mathcal F))^3.$$
This definition makes sense because there exist two constants $\alpha_i(\xi)>0$ ($i=1,2$) such that $\alpha_1(\xi)\|\varphi\|_{(W^1_0(\mathcal F_\xi))^3}\leq 
\|\varphi\circ\varUpsilon\|_{(W^1_0(\mathcal F))^3}\leq \alpha_2(\xi)\|\varphi\|_{(W^1_0(\mathcal F_\xi))^3}$ for all $\varphi\in (W^1_0(\mathcal F_\xi))^3$.
Notice that when $\mathbf f$ is regular enough (i.e. can be identified with a function of $(L^1_{\rm loc}(\mathcal F))^3$) then we get merely $\mathbf f_\xi:=\mathbf f\circ\varUpsilon^{-1}/(J_\xi\circ\varUpsilon^{-1})$. It is easy to check that, according to the properties of $\xi$, the following mapping is a bicontinuous  isomorphism:
$$\begin{array}{rcl}R_\xi:(W^{-1}_0(\mathcal F))^3\times L^2(\mathcal F)\times (H^{1/2}(\mathcal F))^3&\hspace{-0.3cm}\to\hspace{-0.3cm}&
(W^{-1}_0(\mathcal F_\xi))^3\times L^2(\mathcal F_\xi)\times (H^{1/2}(\mathcal F_\xi))^3\\
(\mathbf f,g,\mathbf h)&\hspace{-0.3cm}\mapsto\hspace{-0.3cm}&(\mathbf f_\xi, g_\xi,\mathbf h_\xi),
\end{array}.$$
Denote  $(\mathbf u_\xi,p_\xi)=S_\xi(\mathbf f_\xi, g_\xi,\mathbf h_\xi)$ the unique solution to the Stokes problem \eqref{stokes_varia} in $\mathcal F_\xi$. The operator $S_\xi$ is hence a bicontinuous isomorphism mapping $(W^{-1}_0(\mathcal F))^3\times L^2(\mathcal F)\times (H^{1/2}(\mathcal F))^3$ onto $(W^1_0(\mathcal F_\xi))^3\times L^2(\mathcal F_\xi)$.
The following operator is a bicontinuous isomorphism as well:
$$\begin{array}{rcl}
H_\xi:(W^1_0(\mathcal F_\xi))^3\times L^2(\mathcal F_\xi)&\to&(W^1_0(\mathcal F))^3\times L^2(\mathcal F)\\
(\mathbf v,q)&\mapsto&(\mathbf V,Q)=(\mathbf v\circ\varUpsilon,q\circ\varUpsilon).
\end{array}
$$
The solution to problem \eqref{stokes_change} is provided by the operator $T_\xi:=H_\xi\circ S_\xi\circ R_\xi$ and the following diagram commutes: 
$$
\xymatrix{
    (\mathbf f,g,\mathbf h) \ar[r]^{T_\xi} \ar[d]_{R_\xi} & (\mathbf U_\xi,P_\xi)  \\
    (\mathbf f_\xi, g_\xi,\mathbf h_\xi)  \ar[r]^{S_\xi} & (\mathbf u_\xi,p_\xi) \ar[u]_{H_\xi}
  }
$$
The proof is then completed.
\end{proof}

\bibliographystyle{abbrv}
\bibliography{control_cas_general_biblio}

\begin{thebibliography}{10}

\bibitem{Agrachev:2004aa}
A.~A. Agrachev and Y.~L. Sachkov.
\newblock {\em Control theory from the geometric viewpoint}, volume~87 of {\em
  Encyclopaedia of Mathematical Sciences}.
\newblock Springer-Verlag, Berlin, 2004.

\bibitem{Alouges:2008aa}
F.~Alouges, A.~DeSimone, and A.~Lefebvre.
\newblock Optimal strokes for low {R}eynolds number swimmers: an example.
\newblock {\em J. Nonlinear Sci.}, 18(3):277--302, 2008.

\bibitem{Brenner:1964aa}
H.~Brenner.
\newblock The stokes resistance of a slightly deformed sphere.
\newblock {\em Chemical Engineering Science}, 19(8):519 -- 539, 1964.

\bibitem{chamb_munnier_arxiv}
T.~Chambrion and A.~Munnier.
\newblock Generic controllability of 3d swimmers in a perfect fluid.
\newblock arXiv:1103.5163v1.

\bibitem{Chambrion:2011aa}
T.~Chambrion and A.~Munnier.
\newblock Locomotion and control of a self-propelled shape-changing body in a
  fluid.
\newblock {\em Journal of Nonlinear Science}, 21:325--385, 2011.

\bibitem{Childress:1981aa}
S.~Childress.
\newblock {\em Mechanics of swimming and flying}, volume~2 of {\em Cambridge
  Studies in Mathematical Biology}.
\newblock Cambridge University Press, Cambridge, 1981.

\bibitem{Girault:1991aa}
V.~Girault and A.~Sequeira.
\newblock A well-posed problem for the exterior {S}tokes equations in two and
  three dimensions.
\newblock {\em Arch. Rational Mech. Anal.}, 114(4):313--333, 1991.

\bibitem{Happel_Brenner}
J.~Happel and H.~Brenner.
\newblock {\em Low {R}eynolds number hydrodynamics with special applications to
  particulate media}.
\newblock Prentice-Hall Inc., Englewood Cliffs, N.J., 1965.

\bibitem{Lamb:1993aa}
H.~Lamb.
\newblock {\em Hydrodynamics}.
\newblock Cambridge Mathematical Library. Cambridge University Press,
  Cambridge, sixth edition, 1993.

\bibitem{Lighthill:1975aa}
J.~Lighthill.
\newblock {\em Mathematical biofluiddynamics}.
\newblock Society for Industrial and Applied Mathematics, Philadelphia, Pa.,
  1975.

\bibitem{Lighthill:1952aa}
M.~J. Lighthill.
\newblock On the squirming motion of nearly spherical deformable bodies through
  liquids at very small {R}eynolds numbers.
\newblock {\em Comm. Pure Appl. Math.}, 5:109--118, 1952.

\bibitem{Purcell:1977aa}
E.~M. Purcell.
\newblock Life at low reynolds number.
\newblock {\em American Journal of Physics}, 45(1):3--11, 1977.

\bibitem{Roubicek:2005aa}
T.~Roub{\'{\i}}{\v{c}}ek.
\newblock {\em Nonlinear partial differential equations with applications},
  volume 153 of {\em International Series of Numerical Mathematics}.
\newblock Birkh\"auser Verlag, Basel, 2005.

\bibitem{Shapere:1989aa}
A.~Shapere and F.~Wilczek.
\newblock Geometry of self-propulsion at low {R}eynolds number.
\newblock {\em J. Fluid Mech.}, 198:557--585, 1989.

\bibitem{Simon:1991aa}
J.~Simon.
\newblock Domain variation for drag in stokes flow.
\newblock In X.~Li and J.~Yong, editors, {\em Control Theory of Distributed
  Parameter Systems and Applications}, volume 159 of {\em Lecture Notes in
  Control and Information Sciences}, pages 28--42. Springer Berlin /
  Heidelberg, 1991.

\bibitem{Taylor:1951aa}
G.~Taylor.
\newblock {Analysis of the swimming of microscopic organisms.}
\newblock {\em Proc. R. Soc. Lond., Ser. A}, 209:447--461, 1951.

\bibitem{Whittlesey:1965aa}
E.~F. Whittlesey.
\newblock Analytic functions in {B}anach spaces.
\newblock {\em Proc. Amer. Math. Soc.}, 16:1077--1083, 1965.

\end{thebibliography}
\end{document}